\documentclass[11pt,reqno,a4paper]{amsart}
\usepackage{color}
\usepackage{hyperref,cite}
\usepackage{amsthm}
\usepackage{setspace}
\usepackage{enumerate}
\makeatletter

\numberwithin{equation}{section}
\numberwithin{figure}{section}
\theoremstyle{plain}
\newtheorem{thm}{\protect\theoremname}[section]
  \newtheorem{rem}{\protect\remarkname}[section]
  \newtheorem{lem}{\protect\lemmaname}[section]
  \newtheorem{defn}{\protect\definitionname}[section]

\usepackage{amsfonts}
\usepackage{amsthm}
\usepackage{epsfig}

\@ifundefined{definecolor}
 {\usepackage{color}}{}

\makeatletter
\newcommand{\Rmnum}[1]{\expandafter\@slowromancap\romannumeral#1@}\makeatother

\numberwithin{equation}{section}

\allowdisplaybreaks

\newcommand{\defs}{:=}

\DeclareMathOperator{\di}{div}

\newcommand{\dif}{\mathrm{d}}

\DeclareSymbolFont{lettersA}{U}{pxmia}{m}{it}
\DeclareMathSymbol{\piup}{\mathord}{lettersA}{"19}

\newcommand{\mbnu}{\boldsymbol\nu}
\newcommand{\mbnn}{\mathbf{n}}
\newcommand{\mbtau}{\boldsymbol\tau}
\newcommand{\mbtheta}{\boldsymbol\vartheta}
\newcommand{\mba}{\mathbf{a}}

\makeatother

\makeatother

\usepackage{babel}
  \providecommand{\definitionname}{Definition}
  \providecommand{\lemmaname}{Lemma}
  \providecommand{\remarkname}{Remark}
\providecommand{\theoremname}{Theorem}

\begin{document}

\title[Geometric Structures of Pseudo-Sonic Curves]{Geometric Structures of Pseudo-Sonic Curves\\ in Self-Similar Solutions of \\ the Euler Equations for Potential Flow}

\author[G.-Q. Chen]{Gui-Qiang G. Chen}
\address{Gui-Qiang G. Chen:\, Mathematical Institute, University of Oxford, Oxford, OX2 6GG, UK}
\email{\texttt{gui-qiang.chen@maths.ox.ac.uk}}

\author[M. Feldman]{Mikhail Feldman}
\address{Mikhail Feldman:\, Department of Mathematics,
University of Wisconsin-Madison,
Madison, WI 53706,
USA}
\email{\texttt{feldman@math.wisc.edu}}

\author[X. Gao]{Xin Gao}
\address{Xin Gao: \, School of Mathematical Sciences, Key Laboratory of MEA (Ministry of Education) \& Shanghai Key Laboratory of PMMP, East China Normal University, Shanghai 200241, China}
\email{\texttt{xgao@math.ecnu.edu.cn}}

\author[W. Xiang]{Wei Xiang}
\address{Wei Xiang: \, Department of Mathematics, City University of Hong Kong, Kowloon, Hong Kong, China}
\email{\texttt{weixiang@cityu.edu.hk }}

\begin{abstract}
We are concerned with the geometric structures of pseudo-sonic curves
in two-dimensional self-similar solutions for
the Euler equations for potential flow, allowing for non-uniform supersonic states.
Mathematically, the governing second-order potential flow equation
is of mixed hyperbolic-elliptic type, with degeneracy
occurring along the pseudo-sonic curve.
In this paper, we develop rigorous analytical
approaches to analyze the geometric structures
of pseudo-sonic curves in such self-similar solutions.
We first show that the pseudo-sonic curve is necessarily a circle
if the pseudo-velocity at each point
is a normal to the curve.
We then analyze the general case in which the pseudo-velocity on the pseudo-sonic
point is not a normal to the curve,
and study the geometric properties
of streamlines
in a neighborhood of the pseudo-sonic curve.
Next, we establish two theorems that provide sufficient conditions ensuring that the pseudo-velocity at
a pseudo-sonic point is normal to the curve,
under natural assumptions on the local behavior of the solution.
These results yield a precise characterization of the geometry of pseudo-sonic curves.
Finally, we apply the developed theory to
the shock reflection-diffraction problem with non-uniform incoming flow.
We prove that the pseudo-sonic curve must be an arc if the solution is a $C^2$-small perturbation,
either in the pseudo-supersonic or pseudo-subsonic region, of a solution with uniform
incoming flow (as constructed in \cite{Chen-Feldman4}).
In particular, the density and velocity must be constant, corresponding to the
radius (with a fixed power)  and the center of the pseudo-sonic arc, respectively.
Moreover, we prove that the solution is
$C^{2,\alpha}$-regular in the pseudo-subsonic region up to the sonic arc (except at point $P_1$).
The techniques and ideas developed in this paper are expected to be applicable to other nonlinear
problems involving similar mixed-type degeneracies.
\end{abstract}


\date{\today}
\maketitle

\section{Introduction}
We are concerned with the geometric structures of pseudo-sonic curves
in two-dimensional (2-D) self-similar solutions of the Euler equations for potential flow,
allowing for non-uniform supersonic states.
Such self-similar solutions play a central role in many important applications,
including multidimensional (M-D) Riemann problems and shock reflection-diffraction problem.
They not only arise naturally in a wide range of physical experiments and applications, but also
are fundamental to the mathematical theory of M-D systems of conservation laws,
which remains largely unsolved.

Over the past decades, important progress has been made from a variety of perspectives.
In particular, extensive studies have been devoted to the shock reflection–diffraction problem,
pseudo-supersonic/sonic flow patches, and supersonic flows impinging on solid wedges;
see \cite{Bae-Chen-Feldman,Bae-Chen-Feldman2020,LBERS1950,LBERS1958,ChenDengXiang,Chen-Feldman2,
Chen-Feldman3,Chen-Feldman4,Chen-Feldman5,Chen-Feldman-Hu-Xiang,Chen-Feldman-Xiang,ChenSX,
ChenSX2008,ChenSYiC,EL2008,HJC,LaiSheng,ZhangYX2006,Mo} and the references therein.
In particular, the first global existence result for shock reflection-diffraction configurations
in potential flow was established in \cite{Chen-Feldman2} for large wedge angles $\theta_{\rm w}$.
Subsequently, the global existence, optimal regularity, convexity, and uniqueness of solutions
for all wedge angles $\theta_{\rm w}\in(\theta_{\rm c},\frac{\pi}{2})$ were obtained
in \cite{Chen-Feldman4,Bae-Chen-Feldman,Chen-Feldman-Xiang,Chen-Feldman-Xiang2}.
Further developments include the study of Prandtl-Meyer reflection configurations
in \cite{EL2008,Bae-Chen-Feldman2020}, the Lighthill problem in \cite{Chen-Feldman-Hu-Xiang},
the local stability of Mach reflection for pseudo-steady flow in \cite{ChenSX2008},
and the 2-D Riemann problem with four initial shocks in \cite{ChenCliffeHuangLiuWang}.
For self-similar structures involving rarefaction waves,
the method of characteristic decompositions has been successfully
developed and widely applied to the study of pseudo-supersonic/sonic flow patches;
see \cite{LaiSheng,LiZheng,SongWangZheng}.
In the context of 2-D steady potential flow,
some properties of the sonic curves, as well as  the existence, uniqueness, and regularity of a smooth steady transonic flow,
were established
in \cite{WangCP2016,WangCP2019,WangCP2021}.

For the shock reflection–diffraction problem with uniform supersonic incoming flow,
the pseudo-sonic curve is known to be a circle,
and the optimal regularity of solutions across this pseudo-sonic circle
was established in \cite{Bae-Chen-Feldman}.
It is therefore natural to investigate the geometric structure of the pseudo-sonic curve
when the incoming supersonic state is non-uniform.
To the best of our knowledge, no results in this direction are currently available.
A detailed understanding of pseudo-sonic curves is essential
for capturing the behavior of self-similar transonic flows,
which play a fundamental role in compressible fluid dynamics.
In this paper, we analyze the geometric structures of pseudo-sonic curves in more general settings.
As an application, we prove that,  for the regular shock reflection–diffraction problem,
the pseudo-sonic curve
must lie on a circle whenever the solution is a small perturbation,
in either the pseudo-supersonic or pseudo-subsonic region, of
the solution corresponding to the uniform incoming flow,
as constructed in \cite{Chen-Feldman4}.

To analyze the geometric structure of pseudo-sonic curves,
we first prove in Theorem \ref{theorem1} that the pseudo-sonic curve is
a circle whenever the velocity field is everywhere
normal to the curve.
Then we consider the case in which the velocity field is tangent to the pseudo-sonic curve everywhere.
Through a delicate analysis, we establish
two distinct results: Theorem \ref{theorem4} for the case $\partial_{\mathbf{n}}c=0$, and Theorem \ref{theorem3}
for the case $\partial_{\mathbf{n}}c\neq0$, where $\partial_{\mathbf{n}}$ denotes the normal derivative
along the pseudo-sonic curve.
These theorems show that the pseudo-sonic curve must either be a straight line or be convex
when viewed from the pseudo-supersonic region.
Finally, applying the preceding arguments and results, we characterize
the geometric properties of streamlines
near the pseudo-sonic curve in Theorem \ref{streamtheorem}.

Next, we analyze the behavior of solutions near the pseudo-sonic curve
in order to establish Theorems \ref{theoremnew}--\ref{theoremcone}.
Under natural conditions motivated by solution structures with uniform pseudo-supersonic states,
we prove that the velocity field at any pseudo-sonic point $Q\in \Gamma_{\rm sonic}$ must be normal
to the pseudo-sonic curve $\Gamma_{\rm sonic}$.
Consequently, by Theorem \ref{theorem1}, the pseudo-sonic curve $\Gamma_{\rm sonic}$ must be an arc,
even though it is assumed only to be a $C^{1,1}$-curve.
In contrast to Theorems \ref{theorem1}--\ref{theorem3},
no {\it a priori} assumption is imposed in Theorems \ref{theoremnew}--\ref{theoremcone}
that the pseudo-velocity is normal or tangential to $\Gamma_{\rm sonic}$.
A key ingredient in the analysis is the introduction of a new function $\psi^Q$ (see \eqref{psiQQIN}).
Using this function, Theorem \ref{theoremnew} is proved through a careful analysis of the second-order derivatives of  $\psi^Q$
from both sides of the sonic curve $\Gamma_{\rm sonic}$, exploiting the structure of the governing equation
and the geometry of $\Gamma_{\rm sonic}$.
In addition, we prove that \emph{the solution is $C^2$ across $\Gamma_{\rm sonic}$ near a point $Q$ if the pseudo-velocity at $Q\in\Gamma_{\rm sonic}$
is not a normal to $\Gamma_{\rm sonic}$ and  ${\partial}^{\perp}c(Q)=0$}; see Remark \ref{rem:6.1x}.
To establish Theorem \ref{theoremcone}, we carefully construct a
barrier function (see \eqref{PsiQ==}) in the pseudo-subsonic region
under the assumption that the pseudo-velocity at $Q\in\Gamma_{\rm sonic}$ is not a normal to $\Gamma_{\rm sonic}$.
A contradiction argument is then developed to show that, under natural conditions,
the pseudo-velocity at every point $Q\in\Gamma_{\rm sonic}$ must in fact be normal to $\Gamma_{\rm sonic}$.

As an application of the general framework to the shock reflection-diffraction problem with non-uniform supersonic
incoming flow, we establish a connection between the behavior of solutions near the pseudo-sonic curve and
the geometric structure of the pseudo-sonic curve in the regular shock reflection-diffraction configuration.
More precisely, we prove that \emph{the pseudo-sonic curve must be an arc},
if the solution near the pseudo-sonic curve is a $C^2$-perturbation of the solution
constructed in \cite{Chen-Feldman4} in either pseudo-supersonic (hyperbolic) or pseudo-subsonic (elliptic) region.
Furthermore, we show that \emph{the density and velocity must be constant on the sonic arc, with the radius and center of the pseudo-sonic
arc given respectively by the sonic speed and velocity, whenever the solution is a small perturbation of the solutions
constructed in {\rm \cite{Chen-Feldman4}}}; see Remark \ref{rem2.1x}.
As a consequence, the solution in the pseudo-subsonic region is $C^{2,\alpha}$ up to the sonic arc, except point $P_1$; see Remark \ref{rem4.3x}.
To our knowledge, this is the first rigorous result describing the shape of the pseudo-sonic curve $\Gamma_{\rm sonic}$ and
the behavior of solutions near it without assuming \emph{a priori} that $\Gamma_{\rm sonic}$ is an arc.
Indeed, all previous works have focused on uniform incoming flows, in which case the sonic arc is necessarily circular; 
see, for example, 
\cite{Bae-Chen-Feldman,Bae-Chen-Feldman2020,ChenCliffeHuangLiuWang,ChenDengXiang,Chen-Feldman2,Chen-Feldman4,Chen-Feldman-Hu-Xiang,ChenSYiC,EL2008}. 
For non-uniform incoming states, however, the geometry of the sonic arc is far less clear, and it is conceivable that the circular structure may break down. 
This possibility creates significant analytical challenges. 
One of the main results of this paper is to show that the sonic arc remains circular in the important case when the solution is a small perturbation 
of a solution generated by a uniform incoming state.

The rest of the paper is organized as follows:
In \S 2, we present the pseudo-steady potential flow equation, introduce the pseudo-sonic curves,
and state the main theorems concerning the geometric structures of pseudo-sonic curves
and the behavior of streamlines near pseudo-sonic points.
In \S 3, we present the proof of Theorem 2.1 for the case in which all pseudo-sonic points are exceptional.
In \S 4, we establish Theorems 2.2--2.3 for the case in which the pseudo-velocity is tangential to $\Gamma_{\rm sonic}$.
In \S 5, we prove Theorem 2.4 concerning the structure of pseudo-streamlines near $\Gamma_{\rm sonic}$.
In \S 6--\S7, we prove Theorem 2.5 and Theorem 2.6, respectively.
In \S 8, we apply the preceding results to analyze the pseudo-sonic curve in the regular shock reflection-diffraction configuration
with non-uniform supersonic incoming flow.

\section{Self-Similar Potential Flow and Properties of the Pseudo-Sonic Curves}
In this section, we present the pseudo-steady potential flow equation, introduce the pseudo-sonic curves, and state
the main theorems
concerning the geometric structures of pseudo-sonic curves
and the behavior of streamlines near pseudo-sonic points.

\subsection{Self-similar potential flow equation}
Two-dimensional self-similar potential flow is governed by the conservation law of mass and the Bernoulli law
for the density $\rho$ and the pseudo-velocity potential $\varphi$ in the self-similar variables
$(\xi,\eta)\in \mathbb{R}^2$
({\it cf.} \cite{Chen-Feldman2}):
\begin{align}\label{self-similar}
  \di(\rho (|D\varphi|^2, \varphi)D \varphi) + 2\rho (|D\varphi|^2, \varphi)=0,
\end{align}
with
\begin{align}\label{rhoeqself}
  \rho (|D\varphi|^2, \varphi) =(K- \varphi - \frac12 |D\varphi|^2)^{\frac{1}{\gamma-1}},
\end{align}
where $K$ is the Bernoulli constant and
the divergence $\di$ and gradient $D$ are with respect to $(\xi,\eta)$.
In addition, $(u,v)=D\varphi+(\xi,\eta)$ is the velocity field and
$ c^2(\rho) = (\gamma-1)\rho^{\gamma-1}$ is the sonic speed with $\gamma>1$.
Then the sonic speed can be expressed as
  \begin{align}\label{A2}
 c^2 =  c^2 (|D\varphi|^2, \varphi)= (\gamma - 1)\big(K- \varphi - \frac12 |D\varphi|^2\big).
\end{align}

Moreover, \eqref{self-similar}--\eqref{rhoeqself} can be rewritten as in the following form:
\begin{align}
  &(c^2 -U^2) u_{\xi} -UV (u_{\eta} + v_{\xi})+(c^2 -V^2)v_{\eta} =0,\label{1}\\
  & u_{\eta} = v_{\xi},\label{1.1}
\end{align}
which is closed with the Bernoulli law:
\begin{align}\label{2}
  \frac{c^2}{\gamma-1}+ \frac{U^2 + V^2}{2} + \varphi =K,
\end{align}
where the pseudo-velocity $(U, V)$ satisfying
\begin{align}\label{UVDEFS}
  U\defs u-\xi = \varphi_{\xi}, \quad V\defs  v-\eta = \varphi_{\eta}.
\end{align}

Equation \eqref{self-similar} is a nonlinear equation of mixed elliptic-hyperbolic type.
It is elliptic (\emph{i.e.}, subsonic) if and only if
\begin{align}\label{subsonicdef}
  |D\varphi|< c(|D\varphi|^2, \varphi),
\end{align}
which is equivalent to
\begin{align*}
   |D\varphi|< c_* (\varphi) \defs \sqrt{\frac{2(\gamma-1)}{\gamma+1}(K - \varphi)}.
\end{align*}
It is hyperbolic (\emph{i.e.}, supersonic) if and only if
\begin{align}\label{supsonicdef}
  |D\varphi|> c(|D\varphi|^2, \varphi).
\end{align}

The points at which
\begin{align}\label{sonicDvarphi=c}
   |D\varphi|= c(|D\varphi|^2, \varphi)
\end{align}
are called pseudo-sonic points. 
Through this paper, we assume that the set of all pseudo-sonic points forms a curve, called the pseudo-sonic curve. 
The relative interior of the  pseudo-sonic curve is denoted $\Gamma_{\rm sonic}$:
\begin{equation}\label{sonicCurveRelOpen}
\Gamma_{\rm sonic}=\Gamma_{\rm sonic}^0.
\end{equation}
We sometimes write $\Gamma_{\rm sonic}^0$ to emphasize that only interior points are under consideration.

Denote by $\Omega^-$ and $\Omega^+$ the pseudo-supersonic region and the pseudo-subsonic region
of a solution $\varphi$, respectively.
Assume that there exist solutions $\varphi^-$ and $\varphi^+$
to equations \eqref{self-similar}--\eqref{rhoeqself} in $\Omega^-$ and $\Omega^+$
respectively (see Fig. \ref{sonic}) such that
\begin{enumerate}
\item In $ \Omega^- $, $ \varphi^-$ is a pseudo-supersonic solution and satisfies
equations \eqref{self-similar}--\eqref{rhoeqself} with
\begin{align}\label{supassc22}
 \varphi^-\in C^2(\overline{\Omega^-}),\quad \partial_{\xi} \varphi^- = U_-, \quad \partial_{\eta} \varphi^- = V_- ;
\end{align}

\item In $ \Omega^+ $, $ \varphi^+$ is a pseudo-subsonic solution and satisfies
equations \eqref{self-similar}--\eqref{rhoeqself} 
with
\begin{align}\label{subassc22}
  \varphi^+\in C^2(\Omega^+\cup \Gamma_{\rm sonic}^0),\quad \partial_{\xi} \varphi^+ = U_+, \quad \partial_{\eta} \varphi^+ = V_+.
\end{align}

\item On the pseudo-sonic curve $\Gamma_{\rm sonic}$, the following continuity conditions hold:
\begin{align}\label{conti2}
  \varphi^+ = \varphi^-, \quad D \varphi^+  =D \varphi^- \qquad \text{on $\Gamma_{\rm sonic}$}.
\end{align}
\end{enumerate}
\begin{figure}[!h]
	\centering
	\includegraphics[width=0.5\textwidth]{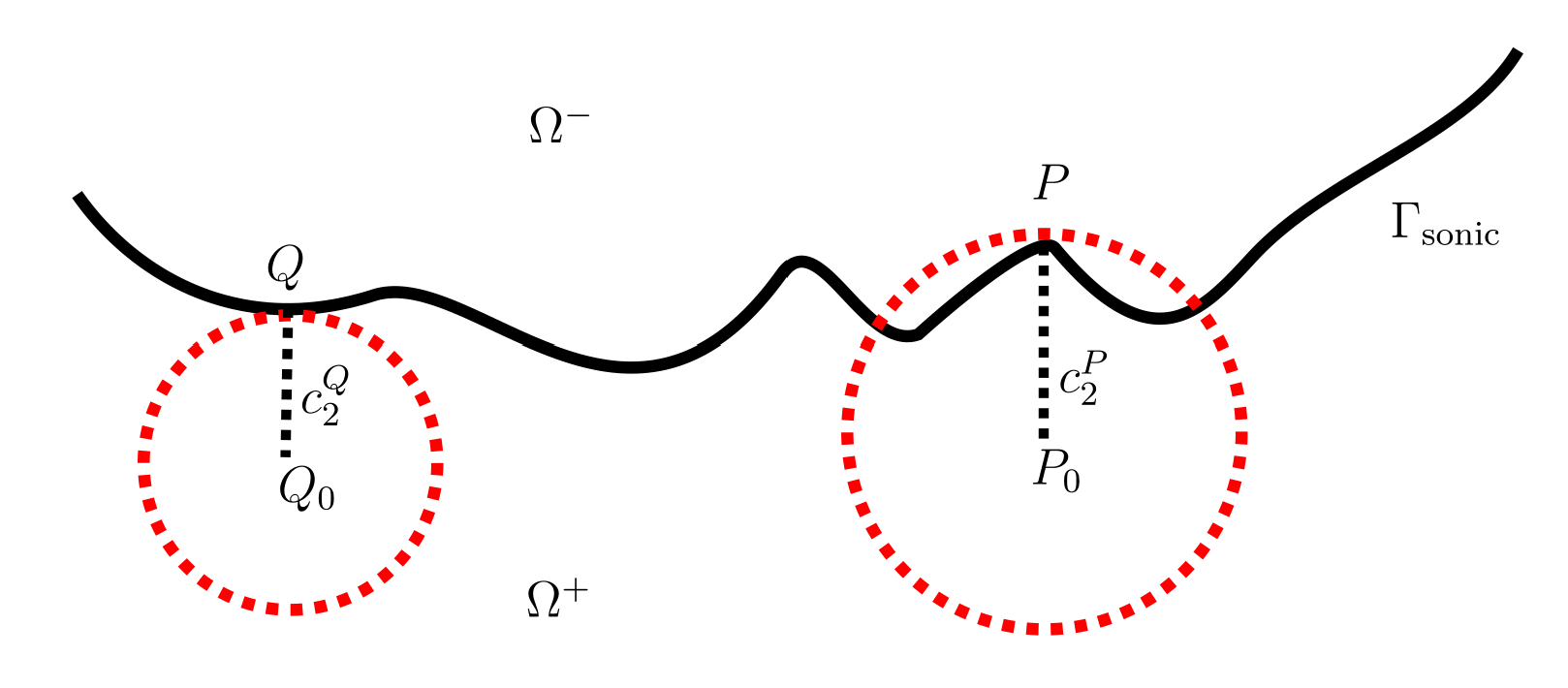}
	\caption{The pseudo-sonic curve in the $(\xi,\eta)$--coordinates. \label{sonic}}
\end{figure}

\subsection{Geometric structures of pseudo-sonic curves}
The first property concerns the case in which all pseudo-sonic points are exceptional
(see Definition \ref{defexcep} below).
Locally, \emph{i.e.}, in a neighborhood of each fixed pseudo-sonic point $Q\in\Gamma_{\rm sonic}$,
assume that one can choose suitable $(\xi,\eta)$--coordinates such that there exists a constant $r>0$
for which the pseudo-sonic curve $\Gamma_{\rm sonic}$ is represented by
\begin{align}\label{soniccurve2}
\Gamma_{\rm sonic}\cap B_r(Q)\defs \big\{ (\xi, \eta)\in \mathbb{R}^2 : \eta =f(\xi)\in C^{1,1}((l_1,l_2))\big\}\cap B_r(Q),
\end{align}
where $l_1<l_2$ are real numbers and $B_r(Q)$ denotes the ball centered at $Q$ with radius $r$.
We then introduce the notion of exceptional points for self-similar potential flows.

\begin{defn}\label{defexcep}
For $f$ given in \eqref{soniccurve2}, if $(U, V)(\xi,f(\xi))$ is the normal vector
to $\Gamma_{\rm sonic}$ at point $(\xi,f(\xi))$, then point $(\xi,f(\xi))$ is called an exceptional point.
\end{defn}

Let
\begin{align}\label{d2.16}
 \partial^{\perp} \defs V \partial_{\xi} -U \partial_{\eta},
 \quad \partial_{\mbtau} \defs \partial_{\xi}
 + f'(\xi) \partial_{\eta},\quad  \partial_{\mbnn} \defs f'(\xi) \partial_{\xi} - \partial_{\eta}.
\end{align}
Then we have the following theorem:

\begin{thm}\label{theorem1}
Assume that \eqref{supassc22}--\eqref{soniccurve2} hold and that the sonic curve is connected. Then
\begin{enumerate}
\item[\rm (i)] A pseudo-sonic point $Q$ is exceptional if and only if ${\partial}_{\mbtau} c(Q)=0$ or $c(Q)=0$.
\item[\rm (ii)] If all pseudo-sonic points are exceptional and the pseudo-sonic speed satisfies $c>0$ on the pseudo-sonic curve,
then $c=c_2$ on the pseudo-sonic curve for some constant $c_2$, and the pseudo-sonic curve lies on a circle of radius $c_2$.
\item[\rm (iii)]
There exists a local solution whose pseudo-sonic curve lies on a circle, and all its sonic points are non-exceptional.
\end{enumerate}
\end{thm}

\begin{rem}\label{rem2.1x}
By {\rm Theorem \ref{theorem1}}, if all pseudo-sonic points are exceptional and $c>0$ on $\Gamma_{\rm sonic}$,
then the pseudo-sonic speed satisfies $c\equiv c_2$ on $\Gamma_{\rm sonic}$ for some constant $c_2$,
and $\Gamma_{\rm sonic}$ lies on a circle.
As shown in {\rm Remark \ref{rem4.3x}} below, it then follows that,
on $\Gamma_{\rm sonic}$
$$
(U,V)=(u_2-\xi,v_2-\eta) \qquad\mbox{for some constants $(u_2, v_2)$}.
$$
Consequently, {\rm Theorem 3.1} in {\rm \cite[page 517]{Bae-Chen-Feldman}} can be applied.
\end{rem}

In the next two theorems, we consider the case in which the pseudo-velocity is tangential to the pseudo-sonic curve.

\begin{thm}\label{theorem4}
Assume that \eqref{supassc22}--\eqref{soniccurve2} hold and that
the pseudo-sonic curve is connected.  Then
\begin{enumerate}
\item[\rm (i)] If the pseudo-velocity field $(U,V)$ is tangential to the pseudo-sonic curve,
then $\partial_\mbnn c =0$ if and only if $\Gamma_{\rm sonic}$ is contained in a straight line.
\item[\rm (ii)] There exists a local solution whose pseudo-sonic curve contained in
a straight line such that
$(U,V)|_{\Gamma_{\rm sonic}}$ is not tangential to the pseudo-sonic line and $\partial_\mbnn c\neq 0$.
\end{enumerate}
\end{thm}

\begin{thm}\label{theorem3}
Assume that \eqref{supassc22}--\eqref{soniccurve2} hold and that the sonic curve is connected.
If the pseudo-velocity field $(U,V)$ is tangential to the pseudo-sonic curve
and $\partial_\mbnn c \neq 0$, then the pseudo-sonic curve is strictly convex
when viewed from the pseudo-supersonic domain $\Omega^-$.
\end{thm}

Next, assume that, in a suitable local coordinate system,
the pseudo-streamlines are represented by
$$
\eta =F(\xi)\in C^2((b_1, b_2)),
$$
where $b_1<b_2$ are real numbers.
Then we have the following result concerning the geometric structure of pseudo-streamlines.

\begin{thm}\label{streamtheorem}
Assume that \eqref{supassc22}--\eqref{soniccurve2} hold.
Let $Q$ be a pseudo-sonic point with $U(Q)\neq0$.  Then
\begin{enumerate}
\item[\rm (i)] If $\partial^{\perp} c(Q) =0$, then $F''(Q)=0$.
\item[\rm (ii)] If $\partial^{\perp} c(Q)\neq 0$ and $F$ can be defined (or extended) on the entire real line
so that $\Omega^- \subset \{\eta > F(\xi) \}$, then $F''(Q)>0$.
\end{enumerate}
\end{thm}

The final two theorems show that a pseudo-sonic point must be exceptional,
whenever the solution satisfies certain natural estimates  near $\Gamma_{\rm sonic}$.
For any fixed pseudo-sonic point $Q\in \Gamma_{\rm sonic}$, set $Q \defs (\xi_*, \eta_*)$ with $\eta_* = f(\xi_*)$.
Let $\varphi_2^Q$ be the pseudo-velocity potential with a constant velocity $(u_2^Q, v_2^Q)=D\varphi(Q)+(\xi_*, \eta_*)$ 
and satisfying $\varphi_2^Q(Q)=\varphi(Q)$
so that
\begin{equation}\label{phi-2-Q}
\varphi_2^Q(\xi,\eta) = -\frac12 \big((\xi^2 + \eta^2) - (\xi_*^2 + \eta_*^2)\big) + u_2^Q(\xi - \xi_*)+ v_2^Q (\eta-\eta_*) + \varphi(Q).
\end{equation}
Furthermore, let $(U,V)(Q) \defs (c\cos \sigma_*, c\sin\sigma_*)$
and define
\begin{align}\label{psiQQIN}
\psi^Q \defs \varphi - \varphi_2^Q.
\end{align}

We now state the following two theorems:

\begin{thm}\label{theoremnew}
Assume that \eqref{supassc22}--\eqref{soniccurve2} hold.
Let ${\psi}^Q$ be defined as in \eqref{psiQQIN}.
Let
$$
\mbnn_Q= \frac{(f'(\xi_*),-1)}{\sqrt{1+(f'(\xi_*))^2}}, \qquad\,\,
\mbtau_Q = \frac{(1,f'(\xi_*))}{\sqrt{1+(f'(\xi_*))^2}}
$$
denote the unit normal and unit tangential vectors to the pseudo-sonic curve at $Q$, respectively.
Suppose that ${\partial}^{\perp}c(Q)=0$.
If there exist a constant $a_1\in(0, \frac{1}{\gamma+1})$
and a sufficiently small constant $\varepsilon>0$ such that
either
\begin{align}\label{npsi<}
\psi^Q(Q + t\mbnn_Q) < a_1 (\frac{\mbnn_Q\cdot(U,V)(Q)}{c(Q)})^2t^2\qquad\,
\mbox{for either $t \in (0, \varepsilon)$ or $t \in (-\varepsilon,0)$},
\end{align}
or
\begin{align}\label{taupsi<}
  \psi^Q(Q + s\mbtau_Q) < a_1 (\frac{\mbtau_Q\cdot(U,V)(Q)}{c(Q)})^2 s^2\qquad\,
  \mbox{for either $s \in (0, \varepsilon)$ or $s \in (-\varepsilon,0)$},
\end{align}
then the pseudo-sonic point $Q$ is exceptional.
\end{thm}

\begin{rem}
For a small perturbation problem of the regular shock reflection, $(U,V)(Q)$ is nearly normal
to the pseudo-sonic curve at point $Q$. Consequently,
$\big(\frac{\mbnn_Q\cdot(U,V)(Q)}{c(Q)}\big)^2$ is close to $1$.
Therefore, condition \eqref{npsi<} is more suitable than \eqref{taupsi<}
for describing such perturbation problems.
An application in this direction is given in {\rm Theorem \ref{thmapp}} below.
\end{rem}

\begin{rem}
Based on the proof of {\rm Theorem \ref{theoremnew}}, and in particular {\rm Lemma \ref{lemntau}} below,
we in fact establish a stronger regularity result. Specifically, if $\psi$ is a piecewise $C^2$ solution $\psi$,
$Q$ is non-exceptional sonic point, and ${\partial}^{\perp}c(Q)=0$,
then $\psi$ is twice differentiable at $Q$ across the sonic curve.
Moreover, if every sonic point is non-exceptional and ${\partial}^{\perp}c=0$ on the entire sonic curve,
then $\psi$ is $C^2$ across the sonic curve.
This is in sharp contrast to the regular shock reflection solutions constructed in {\rm \cite{Chen-Feldman4}},
where one of the second-order derivatives exhibits a jump discontinuity across the sonic curve,
since every sonic point is exceptional.
\end{rem}

\begin{thm}\label{theoremcone}
Fix $\alpha\in(0,1)$.
Assume that \eqref{supassc22}--\eqref{conti2} hold.
Let $Q\in\Gamma_{\rm sonic}$ such that
$c(Q)>0$ and that there exist $r>0$ and a suitable rotated
coordinate system in which $\Gamma_{\rm sonic}$ is represented by
\eqref{soniccurve2} in $B_r(Q)$ for some $f\in C^{1,1}$.
Let $Q=(\xi_*,\eta_*)$ denote the coordinates of $Q$ in this coordinate
system with 
$\eta_*=f(\xi_*)$.
Write
\[
(U,V)(Q)=(c\cos\sigma_*,\,c\sin\sigma_*).
\]
Assume further that $\psi^Q$, defined by \eqref{psiQQIN}, satisfies
\begin{equation}\label{assumeAthmcone}
 \psi^Q
 >
 a_2\big((\xi-\xi_*)\cos\sigma_*+(\eta-\eta_*)\sin\sigma_*\big)^2
 -
 \epsilon\big((\eta-\eta_*)\cos\sigma_*-(\xi-\xi_*)\sin\sigma_*\big)^2
\end{equation}
in $B_r(Q)\cap\Omega^+$ for some constants $a_2>0$ and $\epsilon\ge0$.

Let $\theta_1\in(-\frac{\pi}{2},\frac{\pi}{2})$ be the angle between
the inner normal
\[
\mbnn_Q=
\frac{(f'(\xi_*),-1)}
     {\sqrt{1+(f'(\xi_*))^2}}
\]
to $\Gamma_{\rm sonic}$ at $Q$ and the pseudo-velocity direction
$\mathbf e_v$, where $\mathbf e_v$ and $\mathbf e_{v^\perp}$ are the
unit vectors in the directions of $(U,V)(Q)$ and $(-V,U)(Q)$,
respectively, so that $\cos\theta_1=\mbnn_Q\cdot \mathbf{e}_{v}$ and $\sin\theta_1=\mathbf{n}_Q\cdot \mathbf{e}_{v^{\perp}}$.

Then there exist positive constants $\epsilon_0$ and $\varepsilon_1$,
depending only on
$r$, $a_2$, $\alpha$,
$\|f\|_{C^{1,1}([l_1,l_2])}$, and
$\|\psi^Q\|_{C^{1,\alpha}(\overline{B_r(Q)\cap\Omega^+})}$,
such that, if
\[
0<\epsilon\le\epsilon_0
\qquad\text{and}\qquad
|\theta_1|\le\varepsilon_1,
\]
then the pseudo-sonic point $Q$ is exceptional.
\end{thm}

\begin{rem}
For the regular shock reflection with uniform state $(1)$, it follows from {\rm \cite{Bae-Chen-Feldman}} that, for any pseudo-sonic point $Q\in\Gamma_{\rm sonic}$,
$$
\psi^Q=\frac{1}{2(\gamma+1)}\big((\xi-\xi_*)\cos\sigma_* +(\eta-\eta_*)\sin\sigma_* \big)^2
+
o((\xi-\xi_*)^2 +(\eta-\eta_*)^2).
$$
Then assumption \eqref{assumeAthmcone} is satisfied for some 
constants $a_2=\frac{1}{2(\gamma+1)}$ and $\varepsilon=0$
for such solutions. This suggests that assumption \eqref{assumeAthmcone} should also be satisfied for some constants $a_2\in(0,\frac{1}{2(\gamma+1)})$ 
and sufficiently small $\varepsilon>0$,
provided that the solution is a small perturbation of the solution corresponding to a uniform incident state in $\Omega^+$.
\end{rem}

We now prove Theorems \ref{theorem1}--\ref{theoremcone} in the next five sections.
To simplify the presentation, we henceforth omit the prefix ``pseudo'' from
the terms sonic, supersonic, subsonic, and streamline,
unless otherwise specified.

\section{Proof of Theorem \ref{theorem1} for the Case that All Sonic Points Are Exceptional}
On the sonic curve, in general, two sonic points can have two different sonic circles;
see Fig. \ref{sonic}.
However, if all the sonic points are exceptional, then the following lemma holds:

\begin{lem}\label{arc}
Assume that \eqref{supassc22}--\eqref{soniccurve2} hold.
\begin{enumerate}
\item[(i)]
The sonic point $Q$ is exceptional if and only if either ${\partial}_{\mbtau} c(Q)=0$ or $c(Q)=0$.
\item[(ii)] If all the sonic points are exceptional, $c>0$, and the sonic curve is connected,
then $\Gamma_{\rm sonic}$ lies on a circle with its radius being a constant $c_2$, where $c=c_2$ on $\Gamma_{\rm sonic}$.
\end{enumerate}
\end{lem}

\begin{proof} The proof is divided into three steps.

\smallskip
{\bf 1}. In the first two steps, we consider the argument only locally.
That is, as noted before \eqref{soniccurve2}, we choose
appropriate $(\xi,\eta)$--coordinates, so that  \eqref{soniccurve2} holds, and work within $B_r(Q)$.
Notice that $(U, V)(\xi,f(\xi))$ is a normal vector to the sonic curve $\Gamma_{\rm sonic}$ if and only if
\begin{align}\label{3}
  U+f'(\xi) V=0\qquad \text{on $\Gamma_{\rm sonic}$}.
\end{align}

Substituting \eqref{sonicDvarphi=c} into \eqref{2} yields that, on $\Gamma_{\rm sonic}$,
\begin{align}\label{sonicc}
  \frac{\gamma+1}{2(\gamma-1)}c^2 + \varphi =K.
\end{align}
Taking the tangential derivative ${\partial}_{\mbtau}$ on equation \eqref{sonicc} along $\Gamma_{\rm sonic}$,
we have
 \begin{align*}
   c{\partial}_{\mbtau}c  + \frac{\gamma-1}{\gamma+1} (U + f'\, V) =0\qquad \text{on $\Gamma_{\rm sonic}$}.
 \end{align*}
Thus, according to Definition \ref{defexcep}, the sonic point $Q$ is exceptional if and only if either ${\partial}_{\mbtau} c(Q)=0$ or $c(Q)=0$.

\smallskip
{\bf 2}.
Then it follows from the first step and the assumption that $c>0$ 
on the sonic curve
that $c$ must be a positive constant on the sonic curve.
Let
\begin{equation}\label{c3.33}
c(\xi, f(\xi))= c_0,
\end{equation}
where $c_0>0$ is a constant. Applying \eqref{3} and the identity: $U^2 + V^2 =c^2= c_0^2$ on $\Gamma_{\rm sonic}$,
we have
\begin{align*}
 \big( 1+ (f'(\xi))^2  \big)V^2 (\xi,f(\xi))=U^2 (\xi,f(\xi)) + V^2 (\xi,f(\xi)) =c_0^2.
\end{align*}
This yields that, on $\Gamma_{\rm sonic}$,
\begin{align}
& V(\xi,f(\xi))= \pm \frac{c_0}{\sqrt{1+ (f'(\xi))^2}},\label{VS}\\
&U(\xi,f(\xi))= \mp \frac{ c_0 f'(\xi)}{\sqrt{1+ (f'(\xi))^2}}.\label{US}
\end{align}
On $\Gamma_{\rm sonic}$, \eqref{1} becomes
  \begin{align}\label{UVS}
    V^2 u_{\xi} - UV (u_{\eta} + v_{\xi}) + U^2 v_{\eta} =0.
  \end{align}
Applying \eqref{VS}--\eqref{UVS} yields
  \begin{align}\label{equv}
   u_{\xi} + f'(\xi)u_{\eta} + f'(\xi)v_{\xi} +  (f'(\xi))^2v_{\eta} =0.
  \end{align}
That is, if the tangential derivative along $\Gamma_{\rm sonic}$ is denoted as
$$
\frac{\dif}{\dif\xi}:=\partial_{\xi}+f'(\xi)\partial_{\eta},
$$
then
 \begin{align}\label{equvxi}
   \frac{\dif u}{\dif\xi} + f'(\xi) \frac{\dif v}{\dif\xi} =0.
 \end{align}
Taking the tangential derivative $\frac{\dif}{\dif\xi}$ on equation \eqref{3} along $\Gamma_{\rm sonic}$
leads to
  \begin{align}\label{5}
  \frac{\dif u}{\dif\xi} -1+ f''(\xi) V + f'(\xi) \big( \frac{\dif v}{\dif\xi} -f'(\xi) \big) =0.
 \end{align}
Substituting  \eqref{VS} and \eqref{equvxi} into \eqref{5}, we have
 \begin{align*}
   \pm \frac{c_0}{\sqrt{1+ (f'(\xi))^2}}f''(\xi)-1 -(f'(\xi))^2=0.
 \end{align*}
Then
  \begin{align}\label{fcircle}
    f''(\xi) = \pm \frac{(1 + (f'(\xi))^2)^{\frac32}}{c_0}.
  \end{align}

{\bf 3}. In this step, we show that the sonic curve lies locally on the circle with radius $c_0$.
In fact, if this is true for all points on the sonic curve,
since any two coordinate neighborhoods on the curve have non-empty intersections,
then the two circles for these neighborhoods must coincide,
that is, the sonic curve is on the same circle in the union of the two neighborhoods,
Since the sonic curve is connected and the above fact is shown for arbitrary point on the sonic curve,
we know that the sonic curve is on a circle globally.
Therefore, in what follows, we show that the sonic curve lies on the circle locally.
Then it suffices to solve equation \eqref{fcircle} to show that $\eta =f(\xi)$ lies on the circle
with the center at $(\pm U(0,f(0)), f(0)\pm V(0,f(0)))$ and radius $c_0$.

Let $F(\xi) = f'(\xi)$. Then \eqref{fcircle} yields
\begin{align}
  \frac{ F'(\xi)}{\big(1+F^2(\xi)\big)^{\frac32}} = \pm\frac{1}{c_0}.
\end{align}
It follows that
\begin{align}
  \frac{F(\xi)}{\sqrt{1+F^2(\xi)}} + a = \pm\frac{\xi}{c_0},
\end{align}
where $a = - \frac{F(0)}{\sqrt{1+F^2(0)}}$.
Then
\begin{align}
F(\xi) = \frac{\pm\frac{\xi}{c_0} -a}{\sqrt{1 - (\pm\frac{\xi}{c_0} -a)^2}}.
\end{align}
Thus, we have
\begin{align}\label{f=}
  f(\xi)
  =f(0)+ \int_0^{\xi} \frac{\pm\frac{\tau}{c_0} -a}{\sqrt{1 - (\pm\frac{\tau}{c_0} -a)^2}} \dif \tau
  = \mp c_0 \sqrt{1-(\pm\frac{\xi}{c_0} -a)^2} + b,
\end{align}
where $b = f(0)\pm c_0\sqrt{1-a^2} = f(0)\pm \frac{ c_0}{\sqrt{1+F^2(0)}}$.
Then
\begin{align}
  (f(\xi) - b)^2 = c_0^2 \big(1-(\pm\frac{\xi}{c_0} -a)^2   \big),
\end{align}
that is,
\begin{align}
  (\xi \mp a c_0)^2 + (f(\xi) - b)^2 = c_0^2.
\end{align}
Thus, $(\xi,f(\xi))$ lies on the circle with center at $(\pm ac_0, b)$ and radius $c_0$.
Notice that, by conditions \eqref{3}--\eqref{c3.33},
{\it i.e.}, $U + f'(\xi) V=0$ and $U^2 + V^2 = c_0^2$, we have
\begin{align*}
  ac_0 &= - \frac{F(0)}{\sqrt{1+F^2(0)}} c_0 = - \frac{f'(0)}{\sqrt{1+(f'(0))^2}}c_0 = c_0 \frac{U}{V}\frac{1}{\sqrt{1+ \frac{U^2}{V^2}}}\big|_{(\xi,f(\xi))=(0,f(0))} = U(0,f(0))
\end{align*}
and
\begin{align*}
  b &= f(0)\pm \frac{ c_0}{\sqrt{1+F^2(0)}} = f(0)\pm \frac{ c_0}{\sqrt{1+(f'(0))^2}} = f(0)\pm V(0,f(0)).
\end{align*}

Therefore, $(\xi,f(\xi))$ lies on the circle with center at $(\pm U(0,f(0)), f(0)\pm V(0,f(0)))$ and radius $c_0$.
\end{proof}

On the other hand, if the sonic curve is a circle, $(U,V)$ is not necessarily a normal vector to
the sonic curve $\Gamma_{\rm sonic}$;
that is, the sonic points are not necessarily exceptional.
In fact, a smooth transonic solution near a given sonic circle $\Gamma_{\rm sonic}$, on which any point is not exceptional,
is given in Lemma \ref{sonicarclem} below.
Without loss of the generality, we assume that the sonic circle $\Gamma_{\rm sonic}$
lies on the circle:
\begin{align}\label{ci}
  \xi^2 + \eta^2 =r_0^2,
\end{align}
where $r_0$ is a positive constant.
Let
\begin{align}\label{rthetac}
  (\xi, \eta) = (r\cos\theta, r\sin\theta).
\end{align}
Let
\begin{align}
&\Gamma_{\rm sonic}\defs \big\{(r,\theta)\,:\,r=r_0,\, \theta\in [ \frac{\pi}{6}, \frac{\pi}{4}]\big\},\\
&\mathcal{D}_{\varepsilon} \defs \big\{ (r, \theta)\,:\, |r-r_0|< \varepsilon, \, \theta\in[ \frac{\pi}{6}, \frac{\pi}{4}] \big\}.\label{mathDeq}
\end{align}
We now construct a smooth transonic solution near $\Gamma_{\rm sonic}$.

\begin{lem}\label{sonicarclem}
There exists a smooth solution $(u,v,c)$ satisfying equations \eqref{1}--\eqref{2} in domain $\mathcal{D}_\varepsilon$ with the data{\rm :}
\begin{align}\label{thetauvc}
    u =\frac{2}{\gamma+1} r_0 \cos\theta, \quad v= r_0 \sin\theta, \quad  c= \mu^2 r_0 \cos\theta\qquad\,\,\,\text{on }\Gamma_{\rm sonic},
\end{align}
where $\mu^2 =\frac{\gamma-1}{\gamma+1}$, and the small constant $\varepsilon>0$ depends only on $K$, $\gamma$, and $r_0$.
The solution is supersonic if $r>r_0$ and is subsonic if $r<r_0$.
In addition, any point on $\Gamma_{\rm sonic}$
is not exceptional.
\end{lem}

\begin{proof} In order to obtain a solution in the neighborhood of $\Gamma_{\rm sonic}$,
we use the Cauchy-Kowalevskaya theorem in a small ball near every point of $\Gamma_{\rm sonic}$ up to and including endpoints.
For that, we need to define the data on a slightly larger arc. Therefore, define
\begin{equation}\label{sonicraceq}
\Gamma_{\rm sonic}^{\rm arc}:= \big\{(r,\theta)\,:\,r=r_0,\, \theta\in ( \frac{\pi}{7}, \frac{\pi}{3})\big\},
\end{equation}
and let \eqref{thetauvc} be defined on $\Gamma_{\rm sonic}^{\rm arc}$.
This proof is divided into two steps.

\smallskip
{\bf 1}. In this step, we show that, if the solution exists, then the sonic points are not exceptional and the solution is transonic.
By \eqref{thetauvc},
 \begin{align}\label{UVarc}
   U = u-r_0 \cos\theta = -\mu^2 r_0 \cos\theta, \quad V = v- r_0\sin\theta = 0 \qquad\,\, \text{on $\Gamma_{\rm sonic}^{\rm arc}$}.
 \end{align}
From \eqref{thetauvc}--\eqref{UVarc},  we see that $(U,V)|_{\Gamma_{\rm sonic}^{\rm arc}}$ is not a normal vector to $\Gamma_{\rm sonic}^{\rm arc}$ and
 \begin{align}\label{sonicUVc}
   U^2 + V^2 =c^2 \qquad\,\, \text{on $\Gamma_{\rm sonic}^{\rm arc}$}.
 \end{align}
Thus, any point on the sonic circle $\Gamma_{\rm sonic}^{\rm arc}$ is not exceptional.

We now solve equations \eqref{1}--\eqref{2} in $\mathcal{D}_\varepsilon$ with
data \eqref{thetauvc} by the Cauchy-Kovalevskaya theorem.
It follows from \eqref{1.1} that there exists a function $\varphi$ so that $\varphi_{\xi} = U = u-\xi$ and $\varphi_{\eta} = V= v-\eta$.

Thus, before solving equations \eqref{1}--\eqref{2}, we should check the compatibility condition,
that is, the Bernoulli equation \eqref{2} holds on the pseudo-sonic circle $\Gamma_{\rm sonic}^{\rm arc}$.
In fact, because the tangential derivative along $\Gamma_{\rm sonic}^{\rm arc}$ is $\partial_{\theta}$, the following identity
holds along $\Gamma_{\rm sonic}^{\rm arc}$:
\begin{equation}\label{com:Ber}
\big(\frac{c^2}{\gamma-1}+ \frac{U^2 + V^2}{2} + \varphi \big)_{\theta}=0\qquad\,\,\mbox{along }\Gamma_{\rm sonic}^{\rm arc}.
\end{equation}
It follows from \eqref{thetauvc}--\eqref{UVarc} that
 \[
 u_{\theta} = - c_{\theta} - r_0 \sin\theta =-\frac{2}{\gamma+1}r_0 \sin\theta, \quad v_{\theta} = r_0\cos\theta,\quad
 c_{\theta} = - \mu^2 r_0 \sin\theta\qquad\,\, \text{on $\Gamma_{\rm sonic}^{\rm arc}$}.
 \]
Then
\begin{align}\label{eqqB1}
Uu_{\theta} + Vv_{\theta} + \frac{2}{\gamma-1} cc_{\theta}=0\qquad\,\, \text{on $\Gamma_{\rm sonic}^{\rm arc}$}.
\end{align}
Moreover, on $\Gamma_{\rm sonic}^{\rm arc}$,
\begin{align*}
&\varphi_{\theta}=U(-r_0\sin\theta)+Vr_0\cos\theta=\mu^2r_0^2\sin\theta\cos\theta,\\
&UU_{\theta}=Uu_{\theta}-U(-r_0\sin\theta)=Uu_{\theta}-\mu^2r_0^2\sin\theta\cos\theta.
\end{align*}
Therefore, \eqref{com:Ber} holds.

We now show that the solution of equations \eqref{1}--\eqref{2} with
data \eqref{thetauvc}, if it exists, is transonic.
First, we rewrite equations \eqref{1}-\eqref{2} in  the $(r,\theta)$-coordinates.
Note that
\begin{align}\label{rthetaxieta}
\frac{\partial r}{\partial \xi} = \cos\theta, \quad\frac{\partial r}{\partial \eta} = \sin\theta,
\quad  \frac{\partial \theta}{\partial \xi} = -\frac{\sin\theta}{r},
\quad\frac{\partial \theta}{\partial \eta} = \frac{\cos\theta}{r},
\end{align}
so that
\begin{align}
u_{\xi} =&  \cos\theta u_r -\frac{\sin\theta}{r} u_{\theta},\quad
u_{\eta} = \sin\theta  u_r+  \frac{\cos \theta}{r}u_{\theta},\label{uxietatourtheta}\\
v_{\xi} =&  \cos\theta v_r - \frac{\sin\theta}{r}v_{\theta},\quad
v_{\eta} =  \sin\theta v_r +  \frac{\cos \theta}{r}v_{\theta}.\label{vxietatovrtheta}
\end{align}
Thus, equations \eqref{1}--\eqref{1.1} can be rewritten as
\[
(c^2 -U^2)\big(\cos\theta u_r  - \frac{\sin\theta}{r} u_{\theta}\big) - 2UV \big(\cos\theta  v_r
- \frac{\sin\theta}{r}v_{\theta}\big) +(c^2 -V^2)\big( \sin\theta v_r + \frac{\cos \theta}{r}v_{\theta} \big) =0,
\]
and
\[
\sin\theta u_r + \frac{\cos \theta}{r}u_{\theta}  = \cos\theta   v_r- \frac{\sin\theta}{r}v_{\theta}.
\]
That is,
\begin{equation}
\begin{array}{rl}
&(c^2 -U^2)\cos\theta u_r + \big( (c^2 - V^2)\sin\theta - 2UV \cos\theta  \big)v_r \\[1mm]
&=(c^2 -U^2)\frac{\sin\theta}{r}u_{\theta}- \big( (c^2 - V^2)\frac{\cos\theta}{r} + 2UV \frac{\sin\theta}{r} \big)v_{\theta},\label{1eqq1}
\end{array}
\end{equation}
and
\begin{equation}
\sin\theta u_r - \cos\theta v_r = - \frac{\cos\theta}{r} u_{\theta}- \frac{\sin\theta}{r} v_{\theta}.\label{ueta=vxipolar}
\end{equation}
In addition, taking derivatives on \eqref{2}, we have
\begin{align}
  & Uu_{\xi} + Vv_{\xi} + \frac{2}{\gamma -1} cc_{\xi} =0,\label{Berxi=}\\
  &Uu_{\eta} + Vv_{\eta} + \frac{2}{\gamma -1} cc_{\eta} =0. \label{Bereta=}
\end{align}
Applying \eqref{uxietatourtheta}--\eqref{vxietatovrtheta} and \eqref{Berxi=}--\eqref{Bereta=}, we have
\begin{align}
& U\big(\cos\theta u_r -\frac{\sin\theta}{r} u_{\theta}\big) + V\big( \cos\theta v_r - \frac{\sin\theta}{r}v_{\theta}\big) + \frac{2c}{\gamma -1} \big( \cos\theta c_r -\frac{\sin\theta}{r} c_{\theta}\big) =0,\label{Berthetar1}\\
&U\big( \sin\theta  u_r+  \frac{\cos \theta}{r}u_{\theta} \big) + V\big( \sin\theta v_r +  \frac{\cos \theta}{r}v_{\theta}  \big) + \frac{2c}{\gamma -1} \big( \sin\theta c_r +  \frac{\cos \theta}{r}c_{\theta}  \big) =0. \label{Berthetar2}
\end{align}
Then (\eqref{Berthetar2}$\times \cos\theta$) minus (\eqref{Berthetar1} $\times \sin\theta$) yields
\begin{align}\label{Berthetaeq}
  U u_{\theta} + Vv_{\theta} + \frac{2c}{\gamma -1}c_{\theta} =0,
\end{align}
and (\eqref{Berthetar1} $\times \cos\theta$) plus (\eqref{Berthetar2}$\times \sin\theta$) yields
\begin{align}\label{Berreq}
  Uu_r + Vv_r + \frac{2c}{\gamma -1}c_r=0.
\end{align}

In addition, it follows from \eqref{thetauvc}--\eqref{UVarc}, \eqref{1eqq1}--\eqref{ueta=vxipolar}, and \eqref{Berreq} that, on $\Gamma_{\rm sonic}^{\rm arc}$,
\begin{align}\label{ruveq}
   v_r = - \frac{\cos^2\theta}{\sin\theta}, \quad   u_r = -\frac{\cos\theta}{\sin^2\theta} + \frac{2}{\gamma+1}\cos\theta,\quad
c_r  = \mu^2\cos\theta - \frac{\gamma-1}{2}\frac{\cos\theta}{\sin^2\theta}.
\end{align}

Let
$  \mathcal{F}(\theta,r) \defs  (U^2+V^2-c^2)(\theta,r)$.
Then $ \mathcal{F}(\theta,r_0) = 0$ by \eqref{sonicUVc}.
By \eqref{thetauvc}--\eqref{UVarc} and \eqref{ruveq}, we have
\begin{align*}
    \mathcal{F}_r(\theta,r_0) = 2UU_r+ 2VV_r - 2c c_r = -2c( U_r +  c_r )
= (\gamma-1)r_0 \cot^2\theta >0.
\end{align*}
Therefore,
\begin{align}
    \mathcal{F}(\theta, r) &=0 \qquad \text{when $r = r_0$ (\text{{\it i.e.}, the solution is sonic})},\label{1entropyARC}\\
    \mathcal{F}(\theta,r) &> 0 \qquad \text{when $r> r_0$ (\text{{\it i.e.}, the solution is supersonic})},\\
     \mathcal{F}(\theta,r) &< 0 \qquad \text{when $r<r_0$ (\text{{\it i.e.}, the solution is subsonic})}.\label{entropyARC}
  \end{align}

\smallskip
{\bf 2}. In this step, we prove the existence of solution $(u,v,c)$ to equations \eqref{1}--\eqref{2} in $\mathcal{D}_\varepsilon$ with
data \eqref{thetauvc}.
Let
\begin{align*}
 \psi\defs \varphi+\frac12 r^2,
\qquad\,
(x, y)\defs (r_0-r,\,\theta).
\end{align*}

By \eqref{2} and \eqref{ruveq}, direct calculations yield that \eqref{thetauvc} holds if and only if
\begin{align}\label{noncbry}
  \psi(0,y) = \frac12 \mu^2 r_0^2 (\sin^2 y +\frac{2}{\gamma-1})+K,
  \quad\,\,  \psi_{x}(0,y) =r_0( \mu^2\cos^2y-1).
\end{align}

In fact, using \eqref{2} and $U^2 + V^2 = c^2$ on $\Gamma_{\rm sonic}^{\rm arc}$, we have
\begin{align}
  \frac{\gamma +1}{2(\gamma-1)} c^2 +  \varphi =K\qquad \text{on }  \,\Gamma_{\rm sonic}^{\rm arc}.
\end{align}
Then
\begin{align}
  \varphi(0,y) = K - \frac{\gamma +1}{2(\gamma-1)} c^2 =  K  -  \frac{\gamma +1}{2(\gamma-1)} \mu^4 r_0^2 \cos^2\theta = K  - \frac{\gamma -1}{2(\gamma+1)} r_0^2 \cos^2\theta,
\end{align}
so that
\begin{align}
  \psi(0,y) &= K  - \frac{\gamma -1}{2(\gamma+1)} r_0^2 \cos^2\theta + \frac12 r_0^2\nonumber\\[1mm]
  &= \frac{\gamma -1}{2(\gamma+1)} r_0^2 \big( \frac{2}{\gamma-1} + \sin^2\theta \big) + K\nonumber\\[1mm]
  &= \frac12 \mu^2 r_0^2 \big(\sin^2y +  \frac{2}{\gamma-1}\big) + K.\label{psi0y=}
\end{align}

We now calculate $ \psi_{x}(0,y)$.
By \eqref{Berreq}, we have
\begin{align}\label{varphir=}
  - \varphi_r (r_0, \theta)=\big(\frac{2c}{\gamma-1} c_r + UU_r + VV_r\big)(r_0, \theta).
\end{align}
Then, by \eqref{thetauvc}--\eqref{UVarc} and \eqref{ruveq},
\begin{align*}
 - \varphi_r (r_0, \theta)&= \frac{2}{\gamma-1}\mu^2 r_0 \cos\theta  \big( \mu^2\cos\theta - \frac{\gamma-1}{2}\frac{\cos\theta}{\sin^2\theta} \big)\\[1mm]
&\quad\, -\mu^2 r_0 \cos\theta \big( -\frac{\cos\theta}{\sin^2\theta} + \frac{2}{\gamma+1}\cos\theta - \cos\theta \big)\\[1mm]
& = \mu^2 r_0 \cos^2\theta.
\end{align*}
This implies
\begin{equation}\label{psix0y=}
  \psi_{x}(0,y) = \varphi_x (0,y) - r_0 = \mu^2 r_0 \cos^2\theta - r_0 = r_0 \big( \mu^2\cos^2 y  - 1 \big).
\end{equation}
Thus, \eqref{noncbry} follows. Moreover, by \eqref{sonicUVc}, we know
\begin{align}
  \varphi_{\xi}^2 + \varphi_{\eta}^2 = U^2 + V^2 = c^2\qquad \text{on } \,\Gamma_{\rm sonic}^{\rm arc}.
\end{align}

In the $(x, y)$--coordinates, \eqref{1}--\eqref{2}, \emph{i.e.},
\eqref{1eqq1}--\eqref{ueta=vxipolar} and \eqref{2}, can be rewritten as
\begin{equation}
a_{11} \psi_{xx} + 2 a_{12}\psi_{xy} + a_{22}\psi_{yy} + a_0=0,\label{psixyeq=}
\end{equation}
where
\begin{align}
a_{11}=a_{11}(D\psi, \psi, x,y)\defs&  c^2 - \big(\psi_x + r_0 - x \big)^2,\label{a11xyeq1}\\
a_{12}=a_{12}(D\psi, \psi, x,y)\defs& - \frac{1}{(r_0 - x)^2} (\psi_x +r_0 - x)\psi_y,\label{a12xyeq1}\\
a_{22}=a_{22}(D\psi, \psi, x,y)\defs&\frac{1}{(r_0 - x)^2}\big( c^2 - \frac{1}{(r_0 - x)^2} \psi_{y}^2\big),\label{a22xyeq1}\\
a_0=a_0(D\psi, \psi, x,y)\defs& - \frac{c^2}{r_0 - x}\psi_x - \frac{1}{(r_0 - x)^3} \big(\psi_x + 2(r_0 - x)\big)\psi_y^2, \label{a0xyeq1}\\
c^2 = c^2(D\psi, \psi, x,y) \defs& (\gamma-1)\big(K - \psi - (r_0 - x)\psi_x - \frac12 ( \psi_x^2 + \frac{1}{(r_0 - x)^2} \psi_{y}^2) \big).\label{c2xyeq=}
\end{align}
We notice that equation \eqref{psixyeq=} is degenerate on the sonic arc, since
\begin{align*}
& (a_{11} a_{22} - a_{12}^2 )(0,y)\\
&= \big( c^2 - \big(\psi_x + r_0 \big)^2\big)\frac{1}{r_0^2}\big( c^2 - \frac{1}{r_0^2} \psi_{y}^2\big)(0,y) - \big(\frac{1}{r_0^2} (\psi_x + r_0)\psi_y\big)^2 (0,y)\\
&= \frac{1}{r_0^2}\big(\mu^4 r_0^2 \cos^2 y -  r_0^2\mu^4\cos^4 y\big)\big(\mu^4 r_0^2 \cos^2 y - \frac{1}{r_0^2} r_0^4\mu^4 \sin^2y \cos^2y\big)\\
&\quad\, - \frac{1}{r_0^4}  r_0^2\mu^4\cos^4 y r_0^4\mu^4 \sin^2y \cos^2y\\
&= 0.
\end{align*}

In addition, the sonic arc $\Gamma_{\rm sonic}^{\rm arc}$ defined in \eqref{sonicraceq} becomes
\begin{align}\label{sonicraceqb}
  \widetilde{\Gamma_{\rm sonic}^{\rm arc}} \defs \big\{(x,y)\,:\, x=0,\,y\in( \frac{\pi}{7},\frac{\pi}{3})\big\},
\end{align}
and  $\Gamma_{\rm sonic}$ becomes
\begin{align}
  \widetilde{\Gamma_{\rm sonic}} \defs \big\{(x,y)\,:\, x=0,\,y\in[ \frac{\pi}{6},\frac{\pi}{4}]\big\}.
\end{align}
The coefficient
$a_{11}$
on  $\widetilde{\Gamma_{\rm sonic}^{\rm arc}}$ satisfies 
\begin{align}\label{noncharaeq}
\big(c^2 - \big(\psi_x + (r_0 - x )\big)^2\big)(0,y)=\mu^4 r_0^2 \cos^2 y -  (\mu^2 r_0\cos^2 y)^2
 = \mu^4 r_0^2 \cos^2 y \sin^2 y \neq 0.
\end{align}
Therefore, $\widetilde{\Gamma_{\rm sonic}^{\rm arc}}$ is a non-characteristic boundary.
Then we can use the Cauchy-Kovalevskaya theorem (see \cite[page 229]{Evans}) to prove the existence of
a solution near $\widetilde{\Gamma_{\rm sonic}}$. More precisely, let
\begin{align}\label{newpsi}
\widetilde{\psi}(x, y) \defs \psi(x, y) - \frac12 \mu^2 r_0^2 \big(\sin^2y +  \frac{2}{\gamma-1}\big)-K - r_0 \big( \mu^2\cos^2 y  - 1 \big)x.
\end{align}
By \eqref{noncbry}, it is direct to see that
\begin{align}\label{brytildep}
  \widetilde{\psi}(0,y) =0, \quad  \widetilde{\psi}_{x} (0,y) =0.
\end{align}
In addition,  \eqref{psixyeq=} can be rewritten as
\begin{align}\label{brytildepe}
a_{11}\widetilde{\psi}_{xx} + 2 a_{12}\widetilde{\psi}_{xy} + a_{22} \widetilde{\psi}_{yy} + \widetilde{a}_0=0,
\end{align}
where
\begin{align*}
 \widetilde{a}_0 =  a_{22}\mu^2 r_0 (r_0 - 2x)\cos 2y - 2 a_{12}r_0 \mu^2 \sin 2y + a_0.
\end{align*}
Let
\begin{align*}
  \mbtheta=(\vartheta^1, \vartheta^2, \vartheta^3, \vartheta^4)^{\top}\defs ( \widetilde{\psi}, \widetilde{\psi}_{x}, \widetilde{\psi}_{y}, x)^{\top}.
\end{align*}
Then $\mbtheta$ satisfies the following equations and boundary conditions
\begin{align}
 \mbtheta_{x} =& \mathbf{B}(\mbtheta, y)\mbtheta_{y} + \mba (\mbtheta, y)\qquad \text{for $\sqrt{(y-y_0)^2 + x^2}< \varepsilon$},\label{foeq}\\
 \mbtheta=& {\mathbf{0}} \qquad\qquad\qquad\qquad\qquad
 \text{for $|y-y_0|< \varepsilon$ and $x=0$},\label{bryeqarc}
\end{align}
where $\mathbf{B} = (b^{kl})$ is a $4\times 4$ matrix and $k,l=1,2,3,4$.
Direct calculations yield
\[ \mathbf{B} := \begin{pmatrix} 0& 0 & 0& 0\\
 0& -\frac{2a_{12}(\mbtheta, y)}{a_{11}(\mbtheta, y)} & -\frac{a_{22}(\mbtheta,y)}{a_{11}(\mbtheta, y)} & 0 \\
0& 1 & 0& 0\\
0& 0 & 0& 0
	\end{pmatrix} \]
and $\mba=(a^i)$ is a vector with the following components:
\begin{align}\label{ccde}
   \mba (\mbtheta, y) =& \big( a^1(\mbtheta, y), a^2(\mbtheta, y),a^3(\mbtheta, y),a^4 (\mbtheta, y) \big)^{\top}
   = \big( \vartheta^2, -\frac{\widetilde{a}_0(\mbtheta, y)}{a_{11}(\mbtheta, y)}, 0, 1\big)^{\top}.
\end{align}

It is direct to see that, for any $ y_0\in[\frac{\pi}{6}, \frac{\pi}{4}]$, there exists a constant $r_0>0$ so that
$a_{11}$, $a_{12}$,
$a_{22}$, and $a_0$ are analytic in the ball $B_{\frac{r_0}{2}}(\mathbf{0}, y_0)\subset\mathbb{R}^5$
with center at $(\mathbf{0}, y_0)$ and radius $\frac{r_0}{2}$.
By \eqref{noncharaeq}
and the regularity of $a_{11}$ and $\widetilde{\psi}$, we can choose a constant $r_1\in(0,\frac{r_0}{2})$
small enough such that $a_{11}(\mbtheta, y)>0$ in $B_{r_1}(\mathbf{0}, y_0)\subset\mathbb{R}^5$
for any  $y_0\in[\frac{\pi}{6}, \frac{\pi}{4}]$.
Therefore, $\frac{a_{22}(\mbtheta,y)}{a_{11}(\mbtheta, y)}$, $\frac{a_{12}(\mbtheta,y)}{a_{11}(\mbtheta, y)}$,
and $\frac{\widetilde{a_0}(\mbtheta,y)}{a_{11}(\mbtheta, y)}$ are analytic in $B_{r_1}(\mathbf{0}, y_0)$.
Hence, for any $y_0\in[\frac{\pi}{6}, \frac{\pi}{4}]$,
\begin{equation}\label{3.84x}
\mbox{matrix $\mathbf{B}=(b^{kl})_{4\times4}$ and vector $\mba = (a^i)_{4\times 1}$ are analytic in $B_{r_1}(\mathbf{0}, y_0)$.}
\end{equation}
Then, for any $y_0\in [\frac{\pi}{6}, \frac{\pi}{4}]$, there exists $\varepsilon(y_0)>0$ such that,
in $B_{\varepsilon(y_0)}(0, y_0)$,
the boundary value problem \eqref{foeq}--\eqref{bryeqarc} admits an analytic solution of the following form:
\begin{align}\label{pseq}
    \mbtheta = \sum_{|\alpha|=0}^{\infty} \mbtheta_{\alpha}(y - y_0)^{\alpha_1} x^{\alpha_2},
  \end{align}
where $\mbtheta_{\alpha} = \frac{D^{\alpha} \mbtheta(0,y_0)}{\alpha!}$, $\alpha=(\alpha_1, \alpha_2)$,
and constants $\alpha_1$ and $\alpha_2$ are integers.

Therefore, for any given sonic point $(0, y_0)\in \widetilde{\Gamma_{\rm sonic}}$,
we derive the solution in a small neighborhood of the sonic point $(0, y_0)$.
Apply it in some neighborhood of every point of the closure of $\widetilde{\Gamma_{\rm sonic}}$. By compactness, $\widetilde{\Gamma_{\rm sonic}}$ is covered by a finite subcollection of these neighborhoods. This gives a uniform neighborhood of $\widetilde{\Gamma_{\rm sonic}}$.
We then construct the solution near $\Gamma_{\rm sonic}$.
\end{proof}

Then Theorem \ref{theorem1} follows from Lemmas \ref{arc}--\ref{sonicarclem}.
This completes the proof.

\section{Proof of Theorems \ref{theorem4}--\ref{theorem3} for the Case that the Pseudo-Velocity Is Tangential to the Sonic Curves}
For $f$ given in \eqref{soniccurve2}, according to Definition \ref{defexcep}, if the sonic point is non-exceptional,
\begin{align*}
  U + f'\,V\neq 0.
\end{align*}
Then either vector $(V, -U)(\xi,f(\xi))$ or vector $(-V, U)(\xi,f(\xi))$ points to the supersonic region. That is, there are the following two cases:

\smallskip
\emph{Case I}: Vector $(V, -U)$ points to the supersonic region (see the left-hand side in Fig. \ref{Q1Q2}).
Let $q^2 \defs U^2 + V^2$ and $\partial^{\perp} \defs V \partial_{\xi} -U \partial_{\eta}$.
From \eqref{supsonicdef}--\eqref{sonicDvarphi=c}, we see that $q^2>c^2$ in the supersonic region
and $q^2=c^2$ on the sonic curve, so that
\begin{align}\label{towardtosup}
{\partial}^{\perp}(q^2-c^2)|_{\Gamma_{\rm sonic}}
= V \partial_{\xi}(q^2-c^2)|_{\Gamma_{\rm sonic}}  -U\partial_{\eta} (q^2-c^2)|_{\Gamma_{\rm sonic}} \geq 0.
\end{align}
Taking derivative ${\partial}^{\perp}$ on the Bernoulli equation \eqref{2} leads to
\begin{align}\label{pperpc<0}
  \frac12 {\partial}^{\perp}(q^2-c^2) + \frac{\gamma+1}{\gamma-1}c {\partial}^{\perp} c+ {\partial}^{\perp}\varphi =0.
\end{align}
Notice that
\begin{align}\label{perpvar0}
{\partial}^{\perp}\varphi = \varphi_{\xi} V - \varphi_{\eta} U= UV - VU = 0.
\end{align}
Then \eqref{towardtosup}--\eqref{perpvar0} yield that
\begin{align}\label{perpc<00}
{\partial}^{\perp} c|_{\Gamma_{\rm sonic}} \leq 0.
\end{align}
\begin{figure}[!h]
	\centering
	\includegraphics[width=0.75\textwidth]{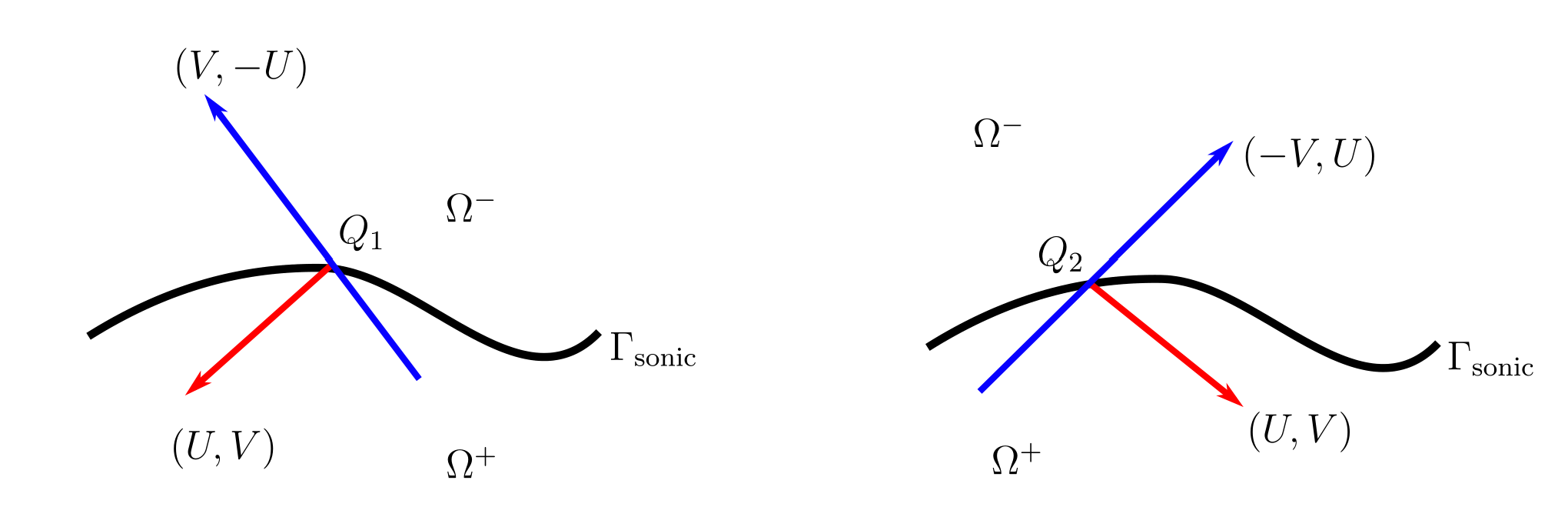}
\caption{}
\label{Q1Q2}
\end{figure}

\smallskip
\emph{Case II}: Vector $(-V, U)$ points to the supersonic region (see the right-hand side in Fig. \ref{Q1Q2}).
It follows from a similar argument as in \eqref{towardtosup} that
\begin{align*}
-\partial^{\perp} (q^2-c^2)_{|\Gamma_{\rm sonic}}= - V \partial_{\xi}(q^2-c^2)|_{\Gamma_{\rm sonic}}  + U\partial_{\eta} (q^2-c^2)|_{\Gamma_{\rm sonic}} \geq 0,
\end{align*}
where  $\partial^{\perp}$ is the same as in Case I so that ``$-$'' appears as above.
By a similar calculation as for \eqref{pperpc<0}--\eqref{perpc<00}, we have
\begin{align}\label{caseII}
{\partial}^{\perp} c|_{\Gamma_{\rm sonic}} \geq 0.
\end{align}

Without loss of generality, we consider only \textit{Case I} in the remainder of the proof, since \textit{Case II} can be treated similarly.
Let
\begin{equation}\label{4.7x}
(U,V)= q(\cos\sigma, \sin\sigma).
\end{equation}
First, we have the following lemma at the non-exceptional sonic point:
\begin{lem}\label{lemperpc}
If $(u,v,c)\in C^1(\overline{\Omega^-})$ and $f\in C^{1,1}(\mathbb{R})$, then, at the non-exceptional sonic point,
 \begin{align}\label{perpcsonic}
   {\partial}^{\perp}c= \frac{(\gamma-1)c}{2} \frac{f'\cos\sigma - \sin\sigma + c\sigma'}{\cos\sigma + f'\sin\sigma},
 \end{align}
 where $\partial^{\perp}= V \partial_{\xi} -U \partial_{\eta}$ and $\sigma' = \sigma_{\xi} + f' \sigma_{\eta}$.
\end{lem}

\begin{proof}
We first notice that $U\ne 0 $ since \eqref{soniccurve2} holds locally near the sonic point under our consideration,
and the point is non-exceptional.
Then
  \begin{align}\label{tansigma==}
    \tan\sigma = \frac{V}{U}.
  \end{align}
Differentiating \eqref{tansigma==} with respect to $\xi$ and $\eta$, respectively, we obtain
\begin{align}\label{sigmaxieta}
\sigma_{\xi} = \frac{UV_{\xi} - VU_{\xi}}{q^2}, \qquad  \sigma_{\eta} = \frac{UV_{\eta} - VU_{\eta}}{q^2}.
\end{align}
By \eqref{1.1} and \eqref{sigmaxieta}, we have
\begin{align}\label{sigmaUV}
V\sigma_{\xi} - U\sigma_{\eta} = \frac{1}{q^2}( 2UV u_{\eta} - V^2u_{\xi} - U^2 v_{\eta} + q^2).
\end{align}
Applying \eqref{1.1} and the sonic condition $q^2 = c^2$, it follows from \eqref{1} that,
at any sonic point,
\begin{align}\label{eqonsonic}
V^2u_{\xi} - 2UVu_{\eta} +U^2v_{\eta} =0.
\end{align}
  Substituting \eqref{eqonsonic} into \eqref{sigmaUV} yields
  \begin{align}\label{sigmaUV=1}
    V\sigma_{\xi} - U\sigma_{\eta} = 1.
  \end{align}
  Combining \eqref{sigmaUV=1} with the fact that $\sigma_{\xi} + f'(\xi)\sigma_{\eta} = \sigma'(\xi,f(\xi))$ leads to
\begin{align}\label{sigmaxieta==}
  \sigma_{\xi}(\xi,f(\xi)) = \frac{f' + U\sigma' }{U + f' V},\qquad
  \sigma_{\eta}(\xi,f(\xi)) =\frac{ V\sigma' -1}{U + f' V}.
\end{align}
Applying the derivative  ${\partial}^{\perp}$ to the Bernoulli equation \eqref{2}, using \eqref{1.1},
and invoking \eqref{sigmaxieta}, we have
\begin{align}\label{perpc==}
  {\partial}^{\perp} c &= -\frac{\gamma-1}{2c}(U{\partial}^{\perp}U + V{\partial}^{\perp}V )\notag\\
  &= -\frac{\gamma-1}{2c} (UVU_{\xi} - U^2 U_{\eta} + V^2 V_{\xi} - UVV_{\eta}) \notag\\
  &= \frac{\gamma-1}{2c} q^2 (U \sigma_{\xi} + V\sigma_{\eta}).
\end{align}
Substituting \eqref{sigmaxieta==} into \eqref{perpc==}, we obtain \eqref{perpcsonic}.
\end{proof}

We now analyze the case that $(U,V)$ is a tangential vector of $\Gamma_{\rm sonic}$, \emph{i.e.}, $(f' U - V)|_{\Gamma_{\rm sonic}} = 0$.
First, we have the following lemma for the case that ${\partial}^{\perp}c=0$ on $\Gamma_{\rm sonic}$:

\begin{lem}\label{tangen000}
If $(u,v,c)\in C^1(\overline{\Omega^-}), f\in C^{1,1}(\mathbb{R})$, and
  \begin{align}\label{CASE3}
   (f'\, U - V)|_{\Gamma_{\rm sonic}} = 0,
  \end{align}
then ${\partial}^{\perp}c|_{\Gamma_{\rm sonic}} =0$
if and only if
$\Gamma_{\rm sonic}$ lies on a straight line.
\end{lem}

\begin{proof}
Condition $(f'\, U - V)|_{\Gamma_{\rm sonic}}= 0$ yields that
 $(U + f'\,V)|_{\Gamma_{\rm sonic}} \neq 0$. Thus the point is non-exceptional.
Then, if ${\partial}^{\perp}c|_{\Gamma_{\rm sonic}} =0$, applying \eqref{perpcsonic} yields
 \begin{align*}
  (f'\cos\sigma - \sin\sigma + c\,\sigma')|_{\Gamma_{\rm sonic}} =0.
 \end{align*}
Combining this with assumption \eqref{CASE3}, we conclude that $\sigma' =0$.
This implies
 \begin{align*}
   f'' = \sigma' \sec^2\sigma =0,
 \end{align*}
where the first equality follows from \eqref{CASE3}.
Thus, the sonic curve lies on a straight line.

On the other hand, if $\Gamma_{\rm sonic}$ lies on a straight line and $(U,V)$ is the tangential vector of the sonic curve, then $\sigma'=0$
on $\Gamma_{\rm sonic}$.
Thus, it follows from \eqref{perpcsonic} and \eqref{CASE3}
that ${\partial}^{\perp}c|_{\Gamma_{\rm sonic}} =0$.
\end{proof}

However, when $\Gamma_{\rm sonic}$ lies on a straight line, $(U,V)$ is not necessarily tangential to the sonic curve.
Without loss of generality, we may assume that the sonic curve is contained in the straight line $\{\eta=0\}$.
We can show that there exists a transonic solution $(u,v,c)$ near the sonic line such that $(U,V)$ is not tangent to the sonic line $\{\eta=0\}$.
 Let
\begin{align}\label{ssonicraceq}
  \Gamma_{\rm sonic}^{\rm str} \defs \big\{(\xi,\eta)\,: \,\eta= 0,\, \xi\in (\xi_1 ,\xi_2)\big\},
\end{align}
where $\xi_1$ and $\xi_2$ are constants.
For any point $(\xi_0,0)\in \Gamma_{\rm sonic}^{\rm str}$, let
\begin{align}\label{smathDeq}
  \mathcal{D}_{\varepsilon_0} \defs \big\{ (\xi,\eta)\,:\,|\eta|< \varepsilon_0,\, \xi\in (\xi_1 ,\xi_2)  \big\}.
\end{align}

Then we have the following lemma:
\begin{lem}\label{counterexf''=0}
There exists a solution $(u,v,c)$ of equations \eqref{1}--\eqref{2} in domain $\mathcal{D}_{\varepsilon_0}$, satisfying
\begin{align}\label{eta=0s}
    u = (1-\frac14 \mu^2) \xi  + \frac12 C_0,\,\,\,\,
     v = \frac{\sqrt{3}}{4}\mu^2 \xi  -\frac{\sqrt{3}}{2} C_0,\,\,\,\, c = -\frac12 \mu^2 \xi  + C_0\qquad\,\, \text{on $\Gamma_{\rm sonic}^{\rm str}$},
\end{align}
where constant $C_0>0$ is chosen such that $c>0$ and $v<0$,
and the small constant $\varepsilon_0 >0$ depends only on $K$, $\gamma$, and $C_0$.
The solution is supersonic when $\eta>0$ and subsonic when $\eta<0$.
In addition, $(U,V)|_{\Gamma_{\rm sonic}^{\rm str}}$ is not the tangential vector of the sonic line $\Gamma_{\rm sonic}^{\rm str}$.
 \end{lem}

\begin{proof}
Employing \eqref{4.7x} and \eqref{eta=0s}, direct calculations yield that, on the sonic line $\Gamma_{\rm sonic}^{\rm str}$,
\begin{align}
U=&  -\frac14 \mu^2 \xi  + \frac12 C_0,\quad
V= \frac{\sqrt{3}}{4}\mu^2 \xi  -\frac{\sqrt{3}}{2} C_0, \quad \cos\sigma = \frac12, \quad  \sin\sigma = -\frac{\sqrt{3}}{2}, \label{cossinxi0}\\
    u_{\xi} =& 1-\frac14 \mu^2,\quad
     v_{\xi}= \frac{\sqrt{3}}{4}\mu^2,\quad  c_{\xi} = - \frac12 \mu^2.\label{cuvxi}
\end{align}
Then $U|_{\Gamma_{\rm sonic}^{\rm str}}=\frac12c|_{\Gamma_{\rm sonic}}>0$, which implies that
 $\Gamma_{\rm sonic}^{\rm str}$ is a non-characteristic boundary.

Applying the Cauchy-Kovalevskaya theorem,
we can obtain the existence of solution $(u,v,c)$ with data \eqref{eta=0s} in $\mathcal{D}_{\varepsilon_0}$,
where $\varepsilon_0$ depends only on $K$, $\gamma$, and $C_0$.
We omit the details of the proof here, since it is similar to the proof of Lemma \ref{sonicarclem}.
Here, similar to Lemma \ref{sonicarclem}, the detail proof is the one-line compactness argument (assuming that both $\xi_1$ and $\xi_2$ are finite);
see the last paragraph in the proof of Lemma \ref{sonicarclem}.

Based on the assumptions
and \eqref{cossinxi0}, it is direct to check that $(U,V)|_{\Gamma_{\rm sonic}^{\rm str}}$ is not the tangential vector of the sonic line $\Gamma_{\rm sonic}^{\rm str}$ and
\begin{align*}
  U^2 + V^2 =c^2 \qquad \text{on $\Gamma_{\rm sonic}^{\rm str}$}.
\end{align*}
Moreover, if $(u,v,c)$ is a solution of equations \eqref{1}--\eqref{2},
then, on $\Gamma_{\rm sonic}^{\rm str}$,
  \begin{align}
     &u_{\xi}\sin^2\sigma  - 2v_{\xi}\sin\sigma \cos\sigma + v_{\eta}  \cos^2\sigma =0,\label{rconter1}\\
     & u_{\eta}=v_{\xi},\label{irrrotauv}\\
    &u_{\eta}\cos\sigma + v_{\eta}\sin\sigma + \frac{2}{\gamma-1} c_{\eta} =0.\label{rconter3}
  \end{align}
Thus, applying \eqref{cossinxi0}--\eqref{rconter3}, we have
\begin{align}\label{rconterceta}
  v_{\eta} = -\frac34 \mu^2 - 3,\,\,\, u_{\eta} = v_{\xi} = \frac{\sqrt{3}}{4}\mu^2,\,\,\,
  c_{\eta}= -\frac{\sqrt{3}}{2}\mu^2 (1+2\gamma) \qquad\,\, \text{on $\Gamma_{\rm sonic}^{\rm str}$}.
  \end{align}
Let
  \begin{align*}
    \mathcal{G}(\xi, \eta) \defs U^2 + V^2 -c^2.
  \end{align*}
Applying \eqref{cossinxi0} and \eqref{rconterceta}, direct calculations yield that
  \begin{align*}
    \mathcal{G}_{\eta}(\xi, 0)
    = 2c( \cos\sigma\, U_{\eta} + \sin\sigma\, V_{\eta} -  c_{\eta})(\xi, 0)
    = 2\sqrt{3}(\gamma+1) c
    > 0.
  \end{align*}
  Then
  \begin{align}
    \mathcal{G}(\xi, 0) &=0 \qquad \text{for $\eta=0\,\,$ (\mbox{sonic when $\eta=0$})},\label{1mathcalGE}\\
    \mathcal{G}(\xi,\eta)&>0 \qquad \text{for $\eta>0\,\,$ (\mbox{supersonic when $\eta>0$})},\\
     \mathcal{G}(\xi, \eta)&< 0 \qquad \text{for $\eta<0\,\,$ (\mbox{subsonic when $\eta<0$})}.\label{mathcalGE}
  \end{align}
Therefore, the solution is supersonic when $\eta>0$, subsonic when $\eta<0$, and sonic when $\eta=0$.
\end{proof}

Theorem \ref{theorem4} follows clearly from Lemmas \ref{tangen000}--\ref{counterexf''=0}.

\medskip
Next, we consider the non-exceptional sonic points with $ {\partial}^{\perp} c|_{\Gamma_{\rm sonic}} <0$
and $(U, V)$ being the tangential vector to the sonic curve to show Theorem \ref{theorem3}.
In fact, we have the following lemma about the convexity of the sonic curve:

\begin{lem}\label{con}
 If $(u,v,c)\in C^1(\overline{\Omega^-}), f\in C^{1,1}(\mathbb{R})$,
 \begin{align}\label{CASE4}
 (f'\, U - V)|_{\Gamma_{\rm sonic}}
 = 0,\qquad {\partial}^{\perp}c|_{\Gamma_{\rm sonic}} < 0,
 \end{align}
then $\sin\sigma$ is strictly decreasing along $\Gamma_{\rm sonic}$
in the tangential direction $(1,f')$.
In addition, the sonic curve is convex when viewed from the supersonic domain $\Omega^-$;
that is, for every point $Q\in \Gamma_{\rm sonic}$, there exists $r>0$ such that $\Omega^- \cap B_r(Q)$ is convex.
\end{lem}

\begin{proof}
It follows from \eqref{CASE4} that $(U, V)$ is the tangential vector to the sonic curve at each point.
Then we choose the $(\xi,\eta)$--coordinates so that the sonic curve given by \eqref{soniccurve2} satisfies
\begin{align}\label{tansigman}
f' =  \tan\sigma \qquad \text{on $\Gamma_{\rm sonic}$}.
\end{align}
Then, by \eqref{perpcsonic} and \eqref{CASE4}--\eqref{tansigman}, we have
\begin{align}\label{r+-perp}
0> {\partial}^{\perp} c(\xi,f(\xi))  =&\frac{(\gamma-1)c^2\sigma'}{2(\cos\sigma+ f'\sin\sigma)}=\frac{\gamma-1}{2} c^2\sigma'\cos\sigma
   = \frac{\gamma-1}{2} c^2(\sin\sigma)'.
\end{align}
Thus, $(\sin\sigma)'< 0$. This implies
\begin{equation}\label{4.31x}
\mbox{$\sin\sigma$ is strictly decreasing along $\Gamma_{\rm sonic}$
in the tangential direction $(1,f')$.}
\end{equation}

Next, \eqref{perpcsonic} and \eqref{CASE4}--\eqref{tansigman} also yield
\begin{align*}
0>{\partial}^{\perp} c(\xi,f(\xi)) =\frac{\gamma-1}{2} \frac{c^2 f''\cos\sigma}{1+(f')^2}.
\end{align*}
Then $f''\cos\sigma < 0$.
In addition, by \eqref{tansigman} and smoothness of $f$, $\cos\sigma$ cannot change sign within a coordinate neighborhood. This and the fact that $f'' \cos \sigma <0$ show that the sign of $f''$ does not change, and
\begin{align}\label{f''cos}
  \mbox{either $\,\,\cos\sigma < 0$ and $f''>0$ $\quad\,\,$ or $\quad\,\, \cos\sigma > 0$ and $f'' < 0$}.
\end{align}

Since we consider Case I, \emph{i.e.}, vector $(V, -U)$ points to the supersonic region (see the left-hand side in Fig. \ref{Q1Q2}),
the directional derivative $\partial^{\perp} = (V \partial_{\xi}, -U \partial_{\eta})$ points
toward the supersonic domain $\Omega^-$.
Consider the first situation in \eqref{f''cos}, \emph{i.e.}, $\cos\sigma < 0$ and $f''>0$.
Since $(U,V)$ is the tangential vector of the sonic curve, $U= q\cos\sigma <0$, $V= q\sin\sigma$, and \eqref{4.31x},
we see that $(V,-U)$ points to the convex part (see the first figure in Fig. \ref{sonicomega}).
Thus, for Case I, it follows that the sonic curve is convex when viewed from the supersonic
domain $\Omega^-$.
The second case in \eqref{f''cos} can be treated similarly, as shown by the second figure in Fig. \ref{sonicomega}.
\end{proof}

\begin{figure}[!h]
	\centering
	\includegraphics[width=0.75\textwidth]{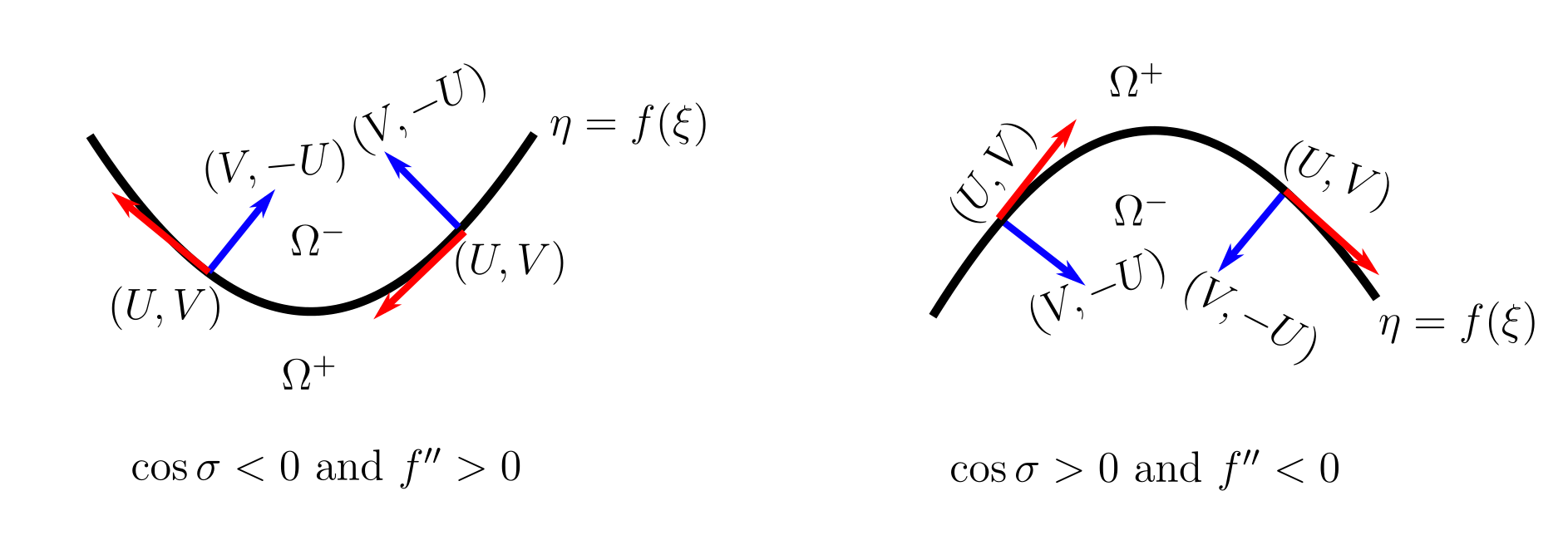}
	\caption{The convexity of the sonic curve. \label{sonicomega}}
\end{figure}

 \begin{rem}
Following a similar argument, it is direct to see that {\rm Lemmas \ref{tangen000}}--{\rm  \ref{con}} still hold for \emph{Case II} on {\rm Fig. \ref{Q1Q2}}.
 \end{rem}

\begin{rem}
In {\rm Theorem \ref{theorem3}}, our assumption is $\partial_nc\neq0$.
Recall that  $\partial^{\perp} = V\partial_{\xi} - U\partial_{\eta}$ and $(U,V)$ is the tangential vector of the sonic curve,
so the directions of the derivatives $\partial_n$ and $ \partial^{\perp} $ are parallel.
Then $0\neq\partial_nc= \partial^{\perp}c$.
\end{rem}

\section{Proof of Theorem \ref{streamtheorem} for the Structure of Streamlines near the Sonic Curves}
Based on Lemmas 3.1--3.2 and Lemmas 4.1--4.4,
the structure of the streamlines near $\Gamma_{\rm sonic}$ can also be described.

\begin{proof}[Proof of {\rm Theorem \ref{streamtheorem}}]
For each sonic point $Q$, we choose the appropriate local coordinates $(\xi, \eta)$ in ball $B_r(Q)$
such that the streamline passing
through $Q$ is given by $\Gamma_{Q}\cap B_r(Q)\defs\{\eta = F(\xi)\in C^2((b_1, b_2))\}\cap B_r(Q)$ with $Q\defs(\xi_Q, F(\xi_Q))$
for some real numbers $b_1$ and $b_2$; that is, $(U,V)$ is the tangential vector along $\Gamma_{Q}$.
Since $U(Q)\neq0$, then
\begin{align}\label{xiFxi==}
  \frac{V}{U}(\xi,F(\xi)) = \frac{\sin\sigma}{\cos\sigma}(\xi,F(\xi)) =F'(\xi) \qquad \text{on $\Gamma_Q\cap B_r(Q)$}.
\end{align}
Direct calculations yield
\begin{align}\label{F''F'}
(U\sigma_{\xi} + V\sigma_{\eta} ) (\xi,F(\xi))
   = \frac{UF''(\xi)}{1+(F'(\xi))^2} \qquad \text{on $\Gamma_Q\cap B_r(Q)$}.
\end{align}
Moreover, it follows from \eqref{perpc==} that
\begin{align}\label{str1}
U \sigma_{\xi} + V\sigma_{\eta} = \frac{2c{\partial}^{\perp} c}{(\gamma-1) q^2}.
\end{align}
Then \eqref{F''F'}--\eqref{str1} imply
\begin{align}\label{F''F'=str1}
  \frac{UF''(\xi)}{1+(F'(\xi))^2} = \frac{2c{\partial}^{\perp} c}{(\gamma-1) q^2}(\xi,F(\xi)) \qquad \text{on $\Gamma_Q\cap B_r(Q)$}.
\end{align}

We now analyze the structure of the streamlines with two cases.

\medskip
(i). ${\partial}^{\perp} c(Q)=0$. In this case, by \eqref{F''F'=str1}, we have
\begin{align}\label{pmcQ=0}
  U(\xi_Q, F(\xi_Q)) F''(\xi_Q)=0.
\end{align}
Since $F(\xi) \in C^2((b_1, b_2))$ and $U(\xi_Q, F(\xi_Q))\neq0$, then \eqref{pmcQ=0} yields that $F''(\xi_Q)=0$.

\begin{figure}[!h]
	\centering
	\includegraphics[width=0.75\textwidth]{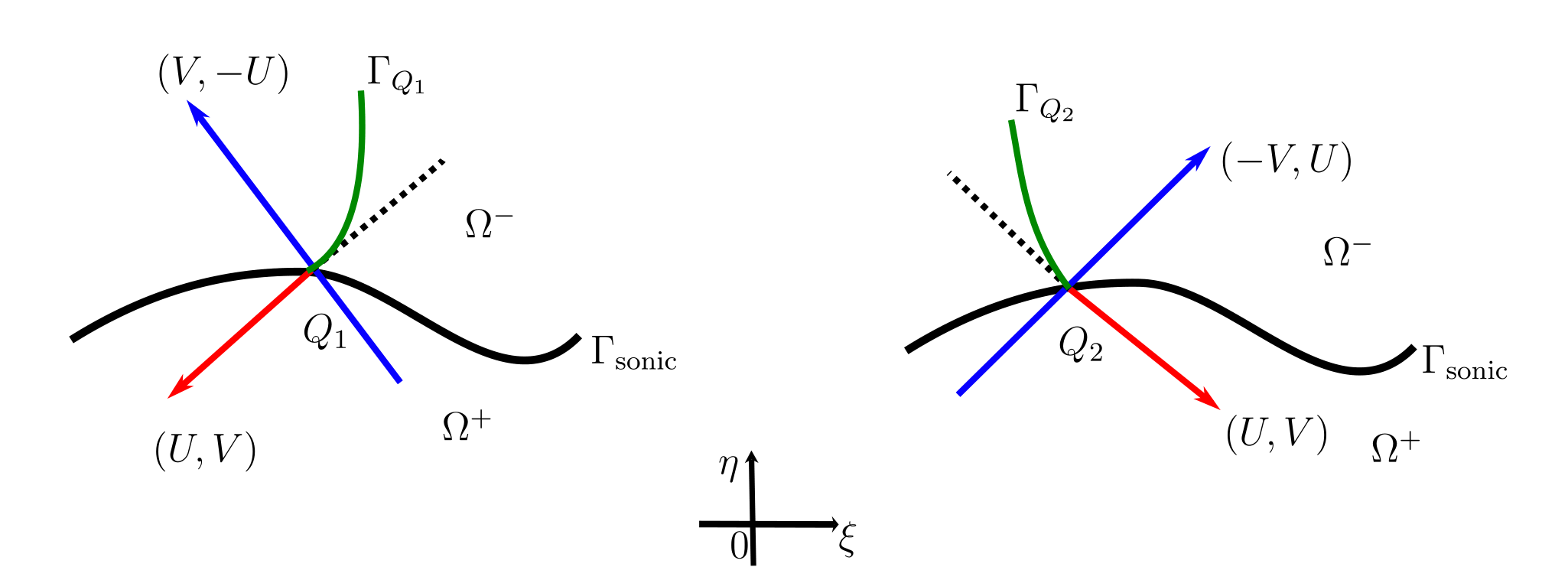}
	\caption{The structure of the streamlines near $\Gamma_{\rm sonic}$. \label{stream}}
\end{figure}

\smallskip
(ii):  ${\partial}^{\perp} c (Q)\neq0$. In this case, it follows
from Lemma \ref{arc} that $Q$ is not exceptional.
In fact, if point $Q$ is exceptional, then $\partial_\tau = \partial^\perp$, which contradicts Lemma \ref{arc}.

\smallskip
First, consider \textit{Case I}, {\it i.e.}, for a non-exceptional point $Q_1$, if
vector $(V, -U)$ points to the supersonic region $\Omega^-$ (see the left-hand side in Fig. \ref{stream}).
It follows from \eqref{perpc<00} and \eqref{F''F'=str1} that
\begin{align}\label{perpc<0}
U F'' (Q_1) < 0.
\end{align}
Therefore, in a small neighborhood of the sonic point $Q_1$, under the condition that
$\Omega^-\subset  \{ \eta> f(\xi)\}$ for the choice of coordinates $(\xi, \eta)$,
the streamline is convex as shown on the left-hand side of Fig. \ref{stream}.

Next, for \textit{Case II}, {\it i.e.}, for a non-exceptional point $Q_2$, vector $(-V, U)$ points to the supersonic region $\Omega^-$,
by a similar argument, \eqref{caseII} and \eqref{F''F'=str1} yield that
\begin{align*}
U F'' (Q_2)>0.
\end{align*}
Thus,  in a small neighborhood of the sonic point $Q_1$, under the condition that
$\Omega^-\subset  \{ \eta> f(\xi)\}$ for the choice of coordinates $(\xi, \eta)$,
the streamline $\Gamma_{Q_2}$ is convex as shown on the right-hand side of Fig. \ref{stream}.
This completes the proof of Theorem \ref{streamtheorem}.
\end{proof}

\section{Proof of Theorem \ref{theoremnew}}
To prove Theorem \ref{theoremnew}, it suffices to prove the following lemma:
\begin{lem}\label{lemntau}
Assume that \eqref{supassc22}--\eqref{soniccurve2} hold.
For any fixed sonic point $Q= (\xi_*,f(\xi_*))$,
let ${\psi}^Q$ be defined as in \eqref{psiQQIN}.
If point $Q$ is non-exceptional and ${\partial}^{\perp}c(Q)=0$, then
${\psi}^Q$ is twice differentiable at $Q$ across the sonic curve and there exist sufficiently small positive constants $\varepsilon$ and $\varepsilon_0$ such that
\begin{align}
&\psi^Q(Q + t\mathbf{n}_Q) \geq  (\frac{1}{\gamma+1} - \varepsilon_0) \big(\frac{\mathbf{n}_Q\cdot(U,V)(Q)}{c(Q)}\big)^2t^2,
\label{PsiQnQ}\\
& \psi^Q(Q + s\mbtau_Q) \geq (\frac{1}{\gamma+1} - \varepsilon_0) \big(\frac{\mbtau_Q\cdot(U,V)(Q)}{c(Q)}\big)^2 s^2,\label{PsiQtauQ}
\end{align}
where $t \in (-\varepsilon,\varepsilon)$, $s \in (-\varepsilon,\varepsilon)$,
$\mathbf{n}_Q = \frac{(f'(\xi_*),-1)}{\sqrt{1+(f'(\xi_*))^2}}$, and $\mbtau_Q = \frac{(1,f'(\xi_*))}{\sqrt{1+(f'(\xi_*))^2}}$.
\end{lem}

\begin{rem}\label{rem:6.1x}
Although $\psi^Q$
is, {\it a priori}, 
a $C^2$ function on $\Omega^+\cup \Gamma_{\rm sonic}$ {\rm (}recall \eqref{sonicCurveRelOpen}{\rm )} and
a $C^2$ function on $\overline{\Omega^-}$, 
it is not necessarily $C^2$ across $\Gamma_{\rm sonic}$. Nevertheless, 
estimates \eqref{PsiQnQ}--\eqref{PsiQtauQ} hold when approached from either side of  $\Gamma_{\rm sonic}$.    
Indeed, under the assumption that the sonic point is non-exceptional and ${\partial}^{\perp}c(Q)=0$,
the second-order derivatives can be given explicitly as in \eqref{imporantcc}--\eqref{PsitaunQ==}.
Thus, the solution is twice differentiable at $Q$ across the sonic curve under the assumptions of {\rm Lemma 6.1}.
Moreover, if the assumptions hold at every point of the sonic curve, then $\psi^Q$ is $C^2$ across the sonic curve.
\end{rem}

\begin{proof}
This proof is divided into two steps.

\smallskip
{\bf 1}. In this step, we calculate $u_{\xi}, u_{\eta} = v_{\xi}, v_{\eta}$ at the non-exceptional sonic point.
By \eqref{1}--\eqref{1.1}, at any sonic point ({\it i.e.}, with $U^2+V^2 = c^2$), we have
\begin{align}\label{eqonsonic1}
V^2u_{\xi} - 2UVu_{\eta} +U^2v_{\eta} =0.
\end{align}
Differentiating the Bernoulli equation \eqref{2} with respect to $\xi$ and $\eta$, respectively, we obtain
\begin{align}
&c_{\xi}=-\frac{\gamma-1}{2c}(U u_{\xi} + Vv_{\xi}),\label{Bernoullixi}\\
& c_{\eta}=-\frac{\gamma-1}{2c}(U u_{\eta} + Vv_{\eta}).\label{Bernoullieta}
\end{align}
In addition, on the sonic curve ({\it i.e.}, $q=c$), the Bernoulli equation becomes
\begin{align}\label{Bcsonic}
  \frac{\gamma+1}{2(\gamma-1)}c^2 + \varphi = K\qquad\mbox{on $\Gamma_{\rm sonic}$}.
\end{align}
Taking tangential derivatives of both sides of \eqref{Bcsonic} along $\Gamma_{\rm sonic}$, we have
\begin{align}\label{ccxietaf'}
  c_{\xi} + f' c_{\eta} = -\mu^2 (\cos\sigma + f'\sin\sigma)\qquad\,\,\mbox{on $\Gamma_{\rm sonic}$}.
\end{align}

We now derive the convenient expressions for $u_\xi$,  $u_\eta=v_\xi$, and $v_\eta$; see \eqref{uxi==}--\eqref{veta==} below.
Recall the definition
\begin{align}
\sin\sigma c_{\xi} - \cos\sigma c_{\eta} = \frac{\partial^{\perp} c }{c}\qquad\,\,\mbox{on $\Gamma_{\rm sonic}$}.\label{perpc=}
\end{align}
Then \eqref{ccxietaf'}$\times \sin\sigma$ minus \eqref{perpc=} yields
\[\label{cxi=}
c_{\eta} = \frac{ -\mu^2\sin\sigma (\cos\sigma + f'\sin\sigma) - \frac{\partial^{\perp} c }{c}}{f'\sin\sigma +\cos\sigma } = \frac{-\partial^{\perp} c }{c (\cos\sigma + f'\sin\sigma)} - \mu^2 \sin\sigma.
\]
Substituting it into \eqref{ccxietaf'}, we have
\[\label{ceta=}
c_{\xi}
= -f'\Big( \frac{-\partial^{\perp} c }{c (\cos\sigma + f'\sin\sigma)} - \mu^2 \sin\sigma  \Big) -\mu^2 (\cos\sigma + f'\sin\sigma)
= \frac{f'\partial^{\perp} c }{c (\cos\sigma + f'\sin\sigma)}- \mu^2 \cos\sigma.
\]

Therefore, if the sonic point $Q$ is non-exceptional, {\it i.e.}, $\cos\sigma + f' \sin\sigma\neq0$,
\begin{align}
   c_{\xi}
   =& \frac{f' \partial^{\perp} c}{c(\cos\sigma + f' \sin\sigma)} -\mu^2 \cos\sigma\qquad\,\,\mbox{on $\Gamma_{\rm sonic}$},\label{cxis}\\
   c_{\eta}
   =&\frac{-\partial^{\perp} c}{c(\cos\sigma + f' \sin\sigma)} -\mu^2 \sin\sigma\qquad\,\,\mbox{on $\Gamma_{\rm sonic}$}. \label{cetas}
 \end{align}

By \eqref{eqonsonic1}--\eqref{Bernoullieta}, $v_{\xi} = u_{\eta}$, and the definition $(U,V)=c(\cos\sigma,\sin\sigma)$, we have
\begin{align}
&\sin^2\sigma u_{\xi} - 2\sin\sigma \cos\sigma u_{\eta} +\cos^2\sigma v_{\eta} =0,\\
&\cos\sigma u_{\xi} + \sin\sigma u_{\eta} = -\frac{2}{\gamma-1}c_{\xi},\\
&\cos\sigma u_{\eta} + \sin\sigma v_{\eta} =-\frac{2}{\gamma-1} c_{\eta}.
\end{align}
It can be rewritten as
\begin{align}
  \mathbf{A} (u_{\xi}, u_{\eta}, v_{\eta})^{\top} =-\big(0, \frac{2}{\gamma-1}c_{\xi}, \frac{2}{\gamma-1} c_{\eta}\big)^{\top},
\end{align}
where
\[ \mathbf{A} := \begin{pmatrix} \sin^2\sigma& -2\sin\sigma \cos\sigma & \cos^2\sigma\\
 \cos\sigma& \sin\sigma & 0 \\
0&\cos\sigma& \sin\sigma
	\end{pmatrix} .\]
Notice that
$\det \mathbf{A} = \sin^4\sigma + 2 \sin^2\sigma \cos^2\sigma  +  \cos^4\sigma  = (\sin^2\sigma + \cos^2\sigma)^2 =1$
so that
\begin{align}\label{uxiuetaveta=}
  (u_{\xi}, u_{\eta}, v_{\eta})^{\top} = - \mathbf{A}^{-1}\big(0, \frac{2}{\gamma-1}c_{\xi}, \frac{2}{\gamma-1} c_{\eta}\big)^{\top},
\end{align}
where
\[ \mathbf{A}^{-1} := \begin{pmatrix} \sin^2\sigma&  \cos\sigma(1+\sin^2\sigma) & -\sin\sigma \cos^2\sigma\\
-\sin\sigma \cos\sigma& \sin^3\sigma & \cos^3\sigma \\
\cos^2\sigma&-\sin^2\sigma \cos\sigma& \sin\sigma(1+\cos^2\sigma)
	\end{pmatrix}. \]
By \eqref{uxiuetaveta=} and \eqref{cxis}--\eqref{cetas}, we have
\begin{align}
  u_{\xi} =&\, -\frac{2}{\gamma-1}\Big( \cos\sigma(1+\sin^2\sigma) c_{\xi} - \sin\sigma \cos^2\sigma\,c_{\eta}   \Big)\notag\\
  =&\, -\frac{2}{\gamma-1}\Big(  \sin\sigma \cos\sigma \frac{\partial^{\perp} c}{c} +  \cos\sigma c_{\xi} \Big)\notag\\
  =&\, -\frac{2}{\gamma-1}\Big(  \sin\sigma \cos\sigma \frac{\partial^{\perp} c}{c} +  \cos\sigma\big(  \frac{f'\partial^{\perp} c }{c (\cos\sigma + f'\sin\sigma)}- \mu^2 \cos\sigma  \big) \Big)\notag\\
  =&\,  -\frac{2\partial^{\perp} c}{(\gamma-1)c}\frac{(\sin\sigma \cos\sigma + f'\sin^2\sigma + f')\cos\sigma}{\cos\sigma + f'\sin\sigma} +\frac{2}{\gamma+1}\cos^2\sigma,\notag\\
  u_{\eta}=&\, -\frac{2}{\gamma-1} \big( \sin^3 \sigma c_{\xi} + \cos^3\sigma c_{\eta}  \big)\notag\\
  =&\, -\frac{2}{\gamma-1} \Big( \sin^3 \sigma \big(   \frac{f'\partial^{\perp} c }{c (\cos\sigma + f'\sin\sigma)}- \mu^2 \cos\sigma \big)
   + \cos^3\sigma \big( \frac{-\partial^{\perp} c }{c (\cos\sigma + f'\sin\sigma)} - \mu^2 \sin\sigma \big)  \Big)\notag\\
  =&\, \frac{2\partial^{\perp} c}{(\gamma-1)c}\frac{\cos^3 \sigma - f' \sin^3 \sigma}{\cos\sigma + f'\sin\sigma} + \frac{2}{\gamma+1} \sin\sigma \cos\sigma,\notag\\
  v_{\eta}=&\, \frac{2}{\gamma-1}\Big( \sin^2\sigma \cos\sigma c_{\xi} - \sin\sigma (1+\cos^2\sigma)c_{\eta}  \Big)\notag\\
  =&\, \frac{2}{\gamma-1}\Big( \sin^2\sigma \cos\sigma \big( \frac{f'\partial^{\perp} c }{c (\cos\sigma + f'\sin\sigma)}- \mu^2 \cos\sigma  \big)\notag\\
    &\qquad\quad\,\,- \sin\sigma (1+\cos^2\sigma)\big( \frac{-\partial^{\perp} c }{c (\cos\sigma + f'\sin\sigma)} - \mu^2 \sin\sigma  \big) \Big)\notag\\
  =&\, \frac{2\partial^{\perp} c}{(\gamma-1)c}\frac{(1+\cos^2\sigma + f'\sin\sigma \cos\sigma )\sin\sigma}{\cos\sigma + f'\sin\sigma}+\frac{2}{\gamma+1}\sin^2\sigma. \notag
\end{align}
Therefore, we obtain
\begin{align}
  u_{\xi} =&\,   -\frac{2\partial^{\perp} c}{(\gamma-1)c}\frac{(\sin\sigma \cos\sigma + f'\sin^2\sigma + f')\cos\sigma}{\cos\sigma + f'\sin\sigma} +\frac{2}{\gamma+1}\cos^2\sigma,\label{uxi==}\\
   u_{\eta}=v_{\xi}=&\, \frac{2\partial^{\perp} c}{(\gamma-1)c}\frac{\cos^3 \sigma - f' \sin^3 \sigma}{\cos\sigma + f'\sin\sigma} + \frac{2}{\gamma+1} \sin\sigma \cos\sigma,\label{ueta==}\\
   v_{\eta}=&\, \frac{2\partial^{\perp} c}{(\gamma-1)c}\frac{(1+\cos^2\sigma + f'\sin\sigma \cos\sigma )\sin\sigma}{\cos\sigma + f'\sin\sigma}+\frac{2}{\gamma+1}\sin^2\sigma.\label{veta==}
\end{align}

{\bf 2}. Applying the continuity condition \eqref{conti2} on the sonic curve, we have
\begin{align}\label{psiQQ}
  \psi^Q(Q) =0, \qquad D\psi^Q(Q)=0.
\end{align}
Then
\begin{align}\label{DpsiQQ}
  \partial_{\mbnu_1} \psi^Q (Q) =0,\qquad  \partial_{\mbnu_2} \psi^Q (Q) =0,
\end{align}
where $\mbnu_1 = (1, f')$ and $\mbnu_2= (f',-1)$ are the tangential vector and
the inner normal vector to the sonic curve, pointing into the subsonic domain $\Omega^+$, respectively.
In addition,
\begin{align}
\partial_{\mbnu_1} \psi^Q (\xi,\eta) &= (u-u_2^Q) + f'(v- v_2^Q),\label{PsiQtau==}\\
\partial_{\mbnu_2} \psi^Q (\xi,\eta) &= f'(u-u_2^Q) - (v- v_2^Q).\label{PsiQn==}
\end{align}
Applying \eqref{uxi==}--\eqref{veta==} and \eqref{PsiQtau==}--\eqref{PsiQn==}, we see that, in either $\overline{\Omega_+}$ or $\overline{\Omega_-}$,
\begin{align}
\partial_{\mbnu_1}^2 \psi^Q
&= u_{\xi}+ 2f'u_{\eta} + (f')^2 v_{\eta}  +  f''(v- v_2^Q) \notag\\
&=  -\frac{2\partial^{\perp} c}{(\gamma-1)c}\frac{(\sin\sigma \cos\sigma + f'\sin^2\sigma + f')\cos\sigma}{\cos\sigma + f'\sin\sigma} +\frac{2}{\gamma+1}\cos^2\sigma \notag\\
&\quad\,+ 2 f' \Big(  \frac{2\partial^{\perp} c}{(\gamma-1)c}\frac{\cos^3 \sigma - f' \sin^3 \sigma}{\cos\sigma + f'\sin\sigma} + \frac{2}{\gamma+1} \sin\sigma \cos\sigma   \Big) \notag\\
 &\quad\, + (f')^2 \Big( \frac{2\partial^{\perp} c}{(\gamma-1)c}\frac{(1+\cos^2\sigma + f'\sin\sigma \cos\sigma )\sin\sigma}{\cos\sigma + f'\sin\sigma}+\frac{2}{\gamma+1}\sin^2\sigma  \Big)+  f''(v- v_2^Q) \notag\\
 &= \frac{2\partial^{\perp} c}{(\gamma-1)c}\frac{\big( \cos\sigma + f'\sin\sigma \big)\big( (f')^2\cos\sigma\sin\sigma + f' \cos^2\sigma - f'\sin^2\sigma -\sin\sigma \cos\sigma\big)}{\cos\sigma + f'\sin\sigma}\notag\\
 &\quad\, + \frac{2}{\gamma+1} \big(\cos\sigma + f' \sin\sigma \big)^2 +  f''(v- v_2^Q)\notag\\
 &=  \frac{2}{\gamma+1} \big(\cos\sigma + f' \sin\sigma \big)^2 +  f''(v- v_2^Q)+  \frac{2\partial^{\perp} c}{(\gamma-1)c}(\cos\sigma + f'\sin\sigma) (f'\cos\sigma - \sin\sigma),\notag
 \end{align}
\begin{align}
 \partial_{\mbnu_2}^2 \psi^Q
 &=(f')^2 u_{\xi}  - 2f'u_{\eta} + v_{\eta} + f'f''(u-u_2^Q)\notag\\
 &=(f')^2\Big(-\frac{2\partial^{\perp} c}{(\gamma-1)c}\frac{(\sin\sigma \cos\sigma + f'\sin^2\sigma + f')\cos\sigma}{\cos\sigma + f'\sin\sigma} +\frac{2}{\gamma+1}\cos^2\sigma \Big) \notag\\
 &\quad\, - 2f'\Big( \frac{2\partial^{\perp} c}{(\gamma-1)c}\frac{\cos^3 \sigma - f' \sin^3 \sigma}{\cos\sigma + f'\sin\sigma} + \frac{2}{\gamma+1} \sin\sigma \cos\sigma \Big)\notag\\
 &\quad\, + \Big( \frac{2\partial^{\perp} c}{(\gamma-1)c}\frac{(1+\cos^2\sigma + f'\sin\sigma \cos\sigma )\sin\sigma}{\cos\sigma + f'\sin\sigma}+\frac{2}{\gamma+1}\sin^2\sigma\Big)+ f'f''(u-u_2^Q)\notag\\
 &=-\frac{2\partial^{\perp} c}{(\gamma-1)c}\frac{(f'\cos\sigma - \sin\sigma)\big(2f' \sin\sigma\cos\sigma + (f')^2 \sin^2\sigma + (f')^2 + 1+ \cos^2\sigma   \big)}{\cos\sigma + f'\sin\sigma}\notag\\
 &\quad\, + \frac{2}{\gamma+1} ( \sin\sigma - f' \cos\sigma)^2 + f'f''(u-u_2^Q)\notag\\
 &=\frac{2}{\gamma+1} ( \sin\sigma - f' \cos\sigma)^2 + f'f''(u-u_2^Q) \notag\\
 &\quad\, -\frac{2\partial^{\perp} c}{(\gamma-1)c} \Big( \cos\sigma + f'\sin\sigma + \frac{1+(f')^2}{\cos\sigma + f'\sin\sigma}  \Big)(f'\cos\sigma - \sin\sigma),
 \end{align}
and
\begin{align}
\partial_{\mbnu_1}\partial_{\mbnu_2} \psi^Q
&=f' u_{\xi} + \big((f')^2-1 \big) u_{\eta} - f' v_{\eta} + f''(u-u_2^Q)\notag\\
&= f' \big(  -\frac{2\partial^{\perp} c}{(\gamma-1)c}\frac{(\sin\sigma \cos\sigma + f'\sin^2\sigma + f')\cos\sigma}{\cos\sigma + f'\sin\sigma} +\frac{2}{\gamma+1}\cos^2\sigma \big)\notag\\
&\quad\, +  \big((f')^2-1 \big) \big( \frac{2\partial^{\perp} c}{(\gamma-1)c}\frac{\cos^3 \sigma - f' \sin^3 \sigma}{\cos\sigma + f'\sin\sigma} + \frac{2}{\gamma+1} \sin\sigma \cos\sigma  \big)\notag\\
&\quad\, - f' \big(  \frac{2\partial^{\perp} c}{(\gamma-1)c}\frac{(1+\cos^2\sigma + f'\sin\sigma \cos\sigma )\sin\sigma}{\cos\sigma + f'\sin\sigma}+\frac{2}{\gamma+1}\sin^2\sigma \big) + f''(u-u_2^Q) \notag\\
&= \frac{2}{\gamma+1}(\cos\sigma + f'\sin \sigma )( f'\cos\sigma - \sin\sigma)+ f'' (u-u_2^Q)\notag\\
&\quad\,-\frac{2\partial^{\perp} c}{(\gamma-1)c}(\cos\sigma + f'\sin \sigma)^2.
 \end{align}

Therefore, we have
\begin{align}
\partial^2_{\mbnu_1} \psi^Q (\xi,\eta)
&= \frac{2}{\gamma+1}(\cos\sigma + f'\sin \sigma )^2 + f'' (v-v_2^Q)\notag\\
  &\quad\, + \frac{2\partial^{\perp}c}{(\gamma-1)c}(\cos\sigma + f'\sin \sigma )(f'\cos\sigma - \sin\sigma),\label{psitt}\\
 \partial^2_{\mbnu_2} \psi^Q (\xi,\eta)
 &= \frac{2}{\gamma+1}(\sin\sigma - f'\cos\sigma )^2 + f' f'' (u-u_2^Q)\notag\\
  &\quad\, -\frac{2\partial^{\perp}c}{(\gamma-1)c}\Big( \cos\sigma + f'\sin \sigma + \frac{1+(f')^2}{\cos\sigma + f'\sin\sigma} \Big) (f'\cos\sigma - \sin\sigma),\label{psinn}\\
   \partial_{\mbnu_1}\partial_{\mbnu_2} \psi^Q (\xi,\eta)
&= \frac{2}{\gamma+1}(\cos\sigma + f'\sin \sigma )( f'\cos\sigma - \sin\sigma)+ f'' (u-u_2^Q)\notag\\
&\quad\, -\frac{2\partial^{\perp} c}{(\gamma-1)c}(\cos\sigma + f'\sin \sigma)^2.\label{Psitaun==}
  \end{align}

At the sonic point $Q= (\xi_*, f(\xi_*))$, $u(Q)= u_2^Q$ and $v(Q)=v_2^Q$.
Thus, under the assumption that $f\in C^{1,1}$,
we see that $ f''(v- v_2^Q)(Q) =0$ and $f'f''(u-u_2^Q)(Q) =0$.
Then, by ${\partial}^{\perp}c(Q)=0$ and \eqref{psitt}--\eqref{Psitaun==}, we have
\begin{align}
\lim_{P\rightarrow Q,\,P\in\overline{\Omega_+}\mbox{ or }P\in\overline{\Omega_-}}\partial^2_{\mbtau} \psi^Q (P)
&= \frac{2}{\gamma+1}\frac{(\cos\sigma_* + f'(\xi_*)\sin \sigma_* )^2}{1+(f'(\xi_*))^2}\notag\\
&= \frac{2}{\gamma+1} \big(\frac{\mbtau_Q\cdot(U,V)(Q)}{c(Q)}\big)^2,\label{imporantcc}\\
\lim_{P\rightarrow Q,\,P\in\overline{\Omega_+}\mbox{ or }P\in\overline{\Omega_-}}\partial^2_{\mathbf{n}} \psi^Q (P)
&= \frac{2}{\gamma+1}\frac{(\sin\sigma_* - f'(\xi_*)\cos\sigma_* )^2}{1+(f'(\xi_*))^2} \notag\\
&= \frac{2}{\gamma+1} \big(\frac{\mathbf{n}_Q\cdot(U,V)(Q)}{c(Q)}\big)^2,\label{imporantcc1}\\
\lim_{P\rightarrow Q,\,P\in\overline{\Omega_+}\mbox{ or }P\in\overline{\Omega_-}} \partial_{\mbtau}\partial_{\mathbf{n}} \psi^Q (P)
&= \frac{2}{\gamma+1}\frac{(\cos\sigma_* + f'(\xi_*)\sin \sigma_* )(f'(\xi_*)\cos\sigma_*-\sin\sigma_*)}{1+(f'(\xi_*))^2} \notag\\
&=\frac{2}{\gamma+1} \frac{(\mbtau_Q\cdot(U,V)(Q))(\mathbf{n}_Q\cdot(U,V)(Q))}{c^2(Q)}.\label{PsitaunQ==}
\end{align}
Since $\psi^Q$ is $C^1$ across the sonic curve, \emph{i.e.}, $(U,V,c)$ is continuous across the sonic curve,
it follows from \eqref{imporantcc}--\eqref{PsitaunQ==} that $\psi^Q$ is twice differentiable at $Q$
across the sonic curve.

Clearly, by the Taylor expansion, \eqref{psiQQ}--\eqref{DpsiQQ}, \eqref{imporantcc}--\eqref{PsitaunQ==},
and the assumption that $\psi^Q\in C^2(\overline{\Omega^{\pm}})$, we conclude \eqref{PsiQnQ}--\eqref{PsiQtauQ}.

\end{proof}

\begin{rem}
Based on {\rm Lemma \ref{lemntau}}, it is direct to see that the sonic point must be exceptional
under the assumptions given in {\rm Theorem \ref{theoremnew}}. Then {\rm  Theorem \ref{theoremnew}} follows.
\end{rem}

\begin{rem}\label{nu1nu2}
Applying \eqref{psiQQ}--\eqref{DpsiQQ}, \eqref{imporantcc}--\eqref{PsitaunQ==}, the Taylor expansion, and
{\rm  Remark \ref{rem:6.1x}}, it follows that, near the sonic point $Q$,
\begin{align}\label{psiQapprox}
  \psi^Q\approx \widehat{\psi^Q}:=\frac{1}{\gamma+1}\Big( \frac{\mathbf{n}_Q\cdot(U,V)(Q)}{c(Q)}(\tilde{\xi} - \tilde{\xi}_*)
  + \frac{\mbtau_Q\cdot(U,V)(Q)}{c(Q)}(\tilde{\eta} - \tilde{\eta}_*) \Big)^2,
\end{align}
where
$\tilde{\xi}\defs f'(\xi_*) \xi - \eta$, $\tilde{\eta}\defs \xi + f'(\xi_*) \eta$, $\tilde{\xi}_*\defs f'(\xi_*) \xi_* - \eta_*$,
and $\tilde{\eta}_*\defs \xi_* + f'(\xi_*) \eta_*$.
From \eqref{psiQapprox}, it follows that $\widehat{\psi^Q}$ is zero along a line near and passing through the sonic point $Q$.
Therefore, \eqref{PsiQnQ}--\eqref{PsiQtauQ} give the optimal estimates near the sonic point $Q$.
\end{rem}

\section{Proof of Theorem \ref{theoremcone}}
We now prove Theorem \ref{theoremcone} by a contradiction argument.

\begin{proof}[Proof of {\rm Theorem \ref{theoremcone}}]
This proof is divided into two steps.

\smallskip
{\bf 1}.
It follows from \eqref{1}--\eqref{1.1} and \eqref{psiQQIN} that $\psi^Q$ satisfies the following equation:
\begin{align}\label{NPSIQ}
  \mathcal{N}(\psi^Q) \defs  a_{11} \psi_{\xi\xi}^Q + 2a_{12}\psi_{\xi\eta}^Q + a_{22}\psi_{\eta\eta}^Q =0,
\end{align}
where
\begin{align*}
  a_{11}= 1-\frac{U^2}{c^2}, \quad a_{12} = -\frac{UV}{c^2},\quad a_{22} = 1-\frac{V^2}{c^2}.
\end{align*}
Since $f\in C^{1,1}(\mathbb{R})$, there is an interior circle $B_{R}(Q_0)\subset\Omega^+$ with center $Q_0=(\xi_0, \eta_0)$ and radius $R$, 
depending only on the $C^{1,1}$-norm of $f$ on $(l_1,l_2)$ and $r$,
which touches $\Gamma_{\rm sonic}$ uniquely at the sonic point $Q=(\xi_*, \eta_*)$ (see Fig. \ref{cone}).
Let $(U,V)(Q) \defs (c\cos \sigma_*, c\sin\sigma_*)$.
Define the $(\tilde{\xi},\tilde{\eta})$--coordinates by
  \begin{align}\label{defcoordinate}
    \tilde{\xi}\defs (\xi-\xi_0)\cos\sigma_* + (\eta-\eta_0)\sin\sigma_*, \quad\,
    \tilde{\eta}\defs -(\xi-\xi_0)\sin\sigma_* + (\eta-\eta_0)\cos\sigma_*.
  \end{align}
Let
\begin{align}\label{widetildepsiQ}
  \widetilde{\psi}^Q(\tilde{\xi},\tilde{\eta})=\psi^Q(\xi,\eta).
\end{align}
Then, in the $(\tilde{\xi},\tilde{\eta})$--coordinates, \eqref{NPSIQ} becomes
  \begin{align*}
   \widetilde{ \mathcal{N}}(\widetilde{\psi^Q})
   := \widetilde{a_{11}} \widetilde{\psi}_{\tilde{\xi}\tilde{\xi}}^Q + 2 \widetilde{a_{12}} \widetilde{\psi}_{\tilde{\xi}\tilde{\eta}}^Q + \widetilde{a_{22}} \widetilde{\psi}_{\tilde{\eta}\tilde{\eta}}^Q =0,
  \end{align*}
  where
  \begin{align*}
  &\widetilde{a_{11}}\defs 1 -\frac{ \widetilde{U}^2}{\widetilde{c}^2},\quad
  \widetilde{a_{12}}\defs -\frac{ \widetilde{U}\widetilde{V}}{\tilde{c}^2},\quad \widetilde{a_{22}}\defs 1 -\frac{ \widetilde{V}^2}{\widetilde{c}^2},\\
  &\widetilde{U} \defs  U\cos\sigma_* + V\sin\sigma_*, \quad \widetilde{V} \defs  V\cos\sigma_* - U\sin\sigma_*,\\
  &\widetilde{c}^2 \defs (\gamma - 1)\big(K- \widetilde{\varphi} - \frac12 |D\widetilde{\varphi}|^2\big).
  \end{align*}
Clearly, $\widetilde{V}(Q)=0$, $\widetilde{U}=\widetilde{c}$, and
\begin{align}
&\widetilde{a_{11}}(Q) =0, \quad \widetilde{a_{12}}(Q)=0, \quad \widetilde{a_{22}}(Q) =1,\label{widetildeaaa}\\
  &\widetilde{a_{11}} >0, \quad \widetilde{a_{22}}>0 \qquad \text{in $\Omega^+$}.\label{a11a22>0}
\end{align}
In addition, in the $(\tilde{\xi}, \tilde{\eta})$--coordinates, $Q_0=(0,0)$ and $Q=(\tilde{\xi}_*, \tilde{\eta}_*)$, where
\begin{align*}
  \tilde{\xi}_*\defs & (\xi_*-\xi_0)\cos\sigma_* + (\eta_*-\eta_0)\sin\sigma_*,\\
   \tilde{\eta}_*\defs & -(\xi_*-\xi_0)\sin\sigma_* + (\eta_*-\eta_0)\cos\sigma_*.
\end{align*}
Thus, \eqref{assumeAthmcone} yields
\begin{equation}\label{assumeAthm2x}
 \widetilde{\psi}^Q>a_2(\tilde{\xi}-\tilde{\xi}_*)^2
 -\epsilon(\tilde{\eta}-\tilde{\eta}_*)^2,
\end{equation}
where $a_2$ is a positive constant and $\epsilon>0$ is a small constant.

\smallskip
{\bf 2}.
Assume that $Q$ is not exceptional, \emph{i.e.}, $(\widetilde{U}(Q),\widetilde{V}(Q))$ is not a normal vector to $\Gamma_{\rm sonic}$ at $Q$.
Then
\begin{align}\label{3.146x}
\tilde{\eta}_*\neq 0.
\end{align}

\begin{figure}[!h]
	\centering
	\includegraphics[width=0.5\textwidth]{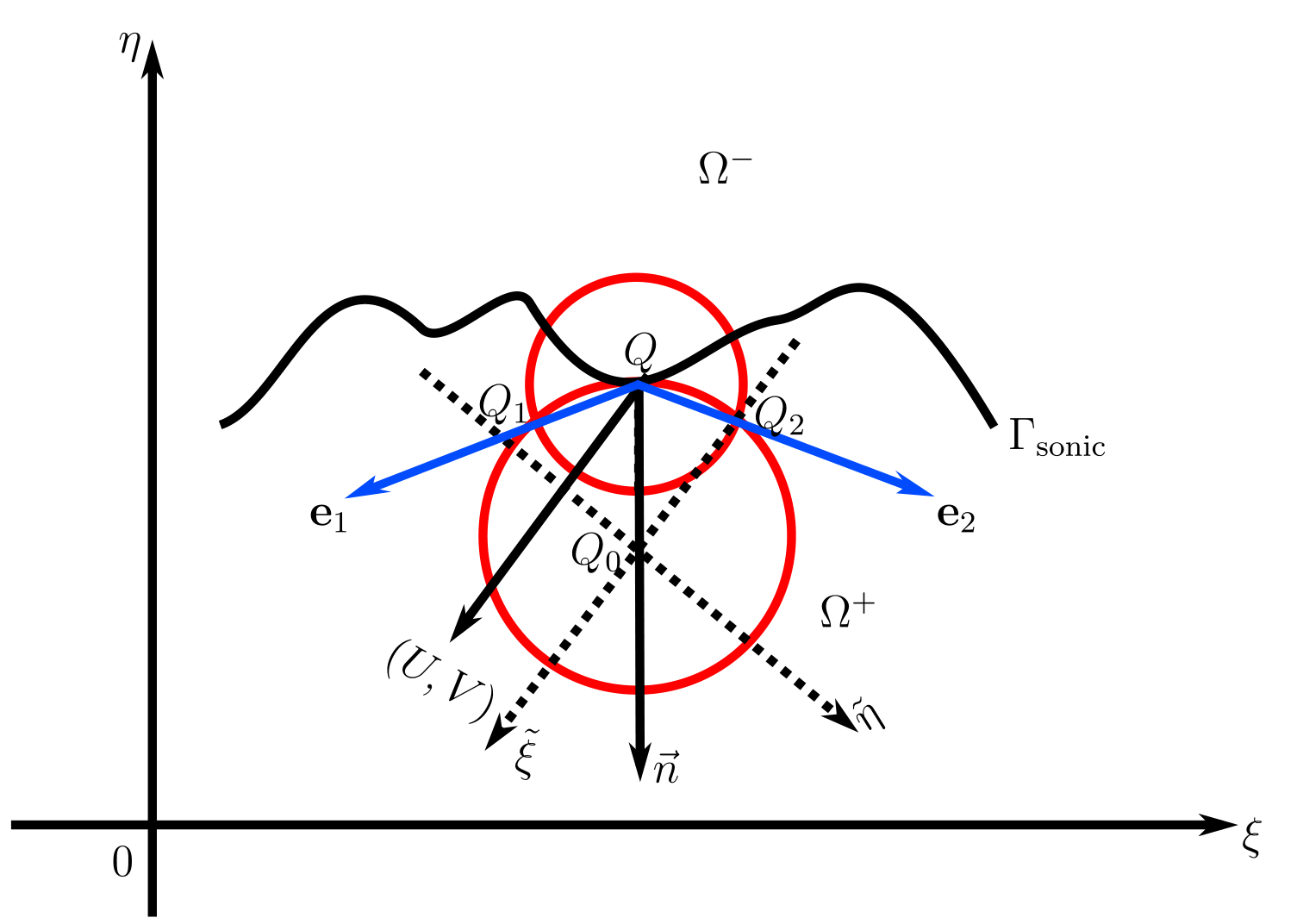}
	\caption{The sonic point $Q$. \label{cone}}
\end{figure}
Let
\begin{align*}
 B_{\varepsilon_*}(Q) =\big\{ (\tilde{\xi},\tilde{\eta})\,:\, (\tilde{\xi}-\tilde{\xi}_*)^2 + (\tilde{\eta} - \tilde{\eta}_*)^2 <\varepsilon_*^2\big\},
\end{align*}
where $\varepsilon_*\in (0, \frac12 |\tilde{\eta}_*|)$ is small such that $Q_0\notin B_{\varepsilon_*}(Q)$,
and let
\begin{align}\label{aij><}
\widetilde{a_{22}} \geq \frac12, \quad |\widetilde{a_{12}}| \leq \frac{|\tilde{\eta}_*|}{100(|\tilde{\xi}_*| + |\tilde{\eta}_*|)},
\quad 0< \widetilde{a_{11}} \leq \frac12 \qquad\,\,\,\mbox{on $B_{\varepsilon_*}(Q)$},
\end{align}
where we have used \eqref{widetildeaaa}--\eqref{a11a22>0} and $\widetilde{\varphi}\in C^2(\overline{\Omega^+})$. 
Based on the construction, $\varepsilon_*$ depends only 
on $R$, $\alpha$, and the $C^{1,\alpha}$-norm of $\psi^{Q}$ in $\overline{B_r(Q)\cap \Omega^+}$.

Let $Q_1$ and $Q_2$ be the two intersection points of $\partial B_{\varepsilon_*}(Q)\cap\partial B_R(Q_0)$,
and let $\mathbf{e}_1=
\vec{QQ_1}$ and $\mathbf{e}_2=
\vec{QQ_2}$ (see Fig. \ref{cone}).
For the given fixed constants $\varepsilon_*$ and $R$,
on arc $\partial B_{\varepsilon_*}(Q)\cap B_R(Q_0)$, we have
\begin{align}\label{theta0*==}
  \tilde{\xi} - \tilde{\xi}_* = \varepsilon_* \cos\theta, \quad \tilde{\eta} - \tilde{\eta}_* = \varepsilon_*\sin\theta
  \qquad\,\, \mbox{for $\theta \in ( - (\theta_0+\theta_1),\,\theta_0-\theta_1)$},
\end{align}
where $\theta_0 = \angle (\mathbf{e}_1, \mathbf{n})= \arccos(\frac{\varepsilon_*}{2R}) \in (0, \frac{\pi}{2})$.
Then we can select $\varepsilon_1$ small enough,
depending only on $\varepsilon_*$ and $R$,
such that, for any $|\theta_1|\leq\varepsilon_1$,
if $\theta \in ( - (\theta_0+\theta_1),  \theta_0-\theta_1)$, then $\cos\theta\geq \frac{\varepsilon_*}{4R}$
and $|\tan\theta|<\frac{4R}{\varepsilon_*}$.
Since $\cos\theta\geq \frac{\varepsilon_*}{4R}$
and $|\tan\theta|<\frac{4R}{\varepsilon_*}$ on $\partial B_{\varepsilon_*}(Q)\cap B_R(Q_0)$,
it follows from \eqref{assumeAthm2x} and \eqref{theta0*==} that, on $\partial B_{\varepsilon_*}(Q)\cap B_R(Q_0)$,
\begin{align*}
\widetilde{\psi}^Q(\tilde{\xi}, \tilde{\eta})&> a_2 \varepsilon_*^2 \cos^2 \theta - \epsilon \varepsilon_*^2 \sin^2 \theta=\varepsilon_*^2\cos^2\theta(a_2-\epsilon\tan^2\theta)\\[1mm]
&> \varepsilon_*^2 \frac{\varepsilon_*^2}{16 R^2}(a_2-\epsilon \frac{16R^2}{\varepsilon_*^2}) = \frac{\varepsilon_*^2}{16 R^2}(a_2\varepsilon_*^2 -16R^2\epsilon ).
\end{align*}
Let
\begin{align}\label{defsepsilon1}
\epsilon_1^*\defs \frac{a_2}{2}  (\frac{\varepsilon_*}{4R})^2.
\end{align}
Clearly,
$\epsilon_1^*$ and $\varepsilon_1$ depend only on $a_2$, $\varepsilon_*$, and $R$, 
and thus they depend only on $a_2$, $r$, the $C^{1,1}$-norm of $f$ on $(l_1,l_2)$,
$\alpha$, and the $C^{1,\alpha}$-norm of $\psi^{Q}$ in $\overline{B_r(Q)\cap \Omega^+}$.
Then, for any $\epsilon<\epsilon_1^*$ and $|\theta_1|\leq\varepsilon_1$, we have
\begin{align}\label{psiQ>kappa}
\widetilde{\psi}^Q(\tilde{\xi}, \tilde{\eta})>  \frac{a_2\varepsilon_*^4}{32R^2}\qquad\,
\text{on $\partial B_{\varepsilon_*}(Q)\cap B_R(Q_0)$}.
\end{align}
Let
\begin{align}\label{PsiQ==}
  \widetilde{\Psi}^Q(\tilde{\xi},\tilde{\eta}) =  \kappa (e^{-l(\tilde{\xi}^2+ \tilde{\eta}^2)} - e^{-l R^2} ) - b (\tilde{\eta} - \tilde{\eta}_*)^2
  \qquad \text{in $\overline{B_{\varepsilon_*}(Q)\cap \Omega^+}$},
\end{align}
where
\begin{align}\label{kappalb=}
  \kappa= \frac{a_2\varepsilon_*^4}{32R^2}, \quad l = \frac{12}{\tilde{\eta}_*^2}, \quad
  b= \frac{l^2}{6} \kappa \tilde{\eta}_*^2 e^{-l((|\tilde{\xi}_*| + |\tilde{\eta}_*|)^2 + 4\tilde{\eta}_*^2)}.
\end{align}
By \eqref{a11a22>0}, \eqref{aij><}, and \eqref{kappalb=},
direct calculations yield that, in domain $B_{\varepsilon_*}(Q)\cap \Omega^+$,
\begin{align}\label{NPsiQ>0}
  \widetilde{\mathcal{N}} \widetilde{\Psi}^Q
  &= 2l \kappa e^{-l(\tilde{\xi}^2+ \tilde{\eta}^2)}\big(2l( \widetilde{a_{11}}\tilde{\xi}^2
  + 2  \widetilde{a_{12}}\tilde{\xi} \tilde{\eta} +  \widetilde{a_{22}} \tilde{\eta}^2)- ( \widetilde{a_{11}} +  \widetilde{a_{22}})\big) -2b  \widetilde{a_{22}}\notag\\
  &\geq  2l \kappa e^{-l(\tilde{\xi}^2+ \tilde{\eta}^2)}\big(2l( 2 \widetilde{a_{12}}\tilde{\xi}\tilde{\eta} +   \widetilde{a_{22}} \tilde{\eta}^2)- ( \widetilde{a_{11}} + \widetilde{a_{22}})\big) -2b  \widetilde{a_{22}}\notag\\
  &= 2l \kappa e^{-l(\tilde{\xi}^2+ \tilde{\eta}^2)} \Big(2l \big( 2 \widetilde{a_{12}}\tilde{\xi} \tilde{\eta}
  +  \frac13 \widetilde{a_{22}} \tilde{\eta}^2 \big)  + \big(2l\frac13 \widetilde{a_{22}} \tilde{\eta}^2- ( \widetilde{a_{11}}
   +  \widetilde{a_{22}})\big)\Big)\notag\\
&\quad\, + 2l  \kappa e^{-l(\tilde{\xi}^2+ \tilde{\eta}^2)}(2l\frac13{ \widetilde{a_{22}}} \tilde{\eta}^2) -2b  \widetilde{a_{22}}\notag\\
&> 2l \kappa e^{-l(\tilde{\xi}^2+ \tilde{\eta}^2)} \Big(l \tilde{\eta}_*^2\big(  \frac{1}{12} - \frac{|\tilde{\xi}|}{50(|\tilde{\xi}_*| + |\tilde{\eta}_*|)} \big)
  + \big((\frac{l}{6}\tilde{\eta}_*^2 -1) \widetilde{a_{22}} -  \widetilde{a_{11}}\big)\Big) \notag\\
&\quad\, + 2\big(\frac{l^2}{6}\kappa \tilde{\eta}_*^2 e^{-l((|\tilde{\xi}_*| + |\tilde{\eta}_*|)^2 + 4 \tilde{\eta}_*^2)} - b \big)  \widetilde{a_{22}}\notag\\
&>2l \kappa e^{-l(\tilde{\xi}^2+ \tilde{\eta}^2)} \big(l \tilde{\eta}_*^2(  \frac{1}{12} - \frac{1}{50} ) + \frac{l\tilde{\eta}_*^2}{24}- \frac12\big)\notag\\
&> 0.
\end{align}

Notice that $(e^{-l(\tilde{\xi}^2+ \tilde{\eta}^2)} - e^{-l R^2} ) \in (0,1)$ on arc $\partial B_{\varepsilon_*}(Q)\cap B_R(Q_0)$,
which implies
\begin{align}\label{PsiQ>kappaon}
   \widetilde{\Psi}^Q \leq  \kappa\qquad\,\, \text{on $\,\partial B_{\varepsilon_*}(Q)\cap B_R(Q_0)$}.
\end{align}
Thus, \eqref{psiQ>kappa}, \eqref{kappalb=}, and \eqref{PsiQ>kappaon} lead to
\begin{align}\label{boundary1=}
   \widetilde{\psi}^Q \geq  \widetilde{\Psi}^Q \qquad \text{on $\partial B_{\varepsilon_*}(Q)\cap B_R(Q_0)$}.
\end{align}

It follows from assumption \eqref{assumeAthm2x} that 
\begin{equation}\label{IandII}
\widetilde{\psi}^Q>-\epsilon(\tilde{\eta} - \tilde{\eta}_*)^2\qquad \mbox{on $\partial (B_{\varepsilon_*}(Q) \cap \Omega^+) \setminus  B_R(Q_0)$}.
\end{equation}

Since
 $\tilde{\xi}^2+ \tilde{\eta}^2 \geq R^2$ on arc  $\partial (B_{\varepsilon_*}(Q) \cap \Omega^+) \setminus B_R(Q_0)$, then
\begin{align}
 (e^{-l(\tilde{\xi}^2+ \tilde{\eta}^2)} - e^{-l R^2} ) \leq 0 \qquad \text{on $\partial(B_{\varepsilon_*}(Q) \cap \Omega^+) \setminus  B_R(Q_0)$}.
\end{align}
By \eqref{PsiQ==}, we have
\begin{equation}\label{Noteleq0}
\widetilde{\Psi}^Q \leq  -b (\tilde{\eta} - \tilde{\eta}_*)^2\qquad \mbox{on } \partial (B_{\varepsilon}(Q) \cap \Omega^+) \setminus B_R(Q_0).
\end{equation}

For the fixed constant $b$ given in \eqref{kappalb=}, let
\begin{align}\label{epsilon2=}
\epsilon_2 \defs \frac b2. 
\end{align}
Clearly, by \eqref{kappalb=}, $\epsilon_2$ depends only on $a_2$, $\varepsilon_*$, and $R$, so that depends only on $a_2$, $r$, 
the $C^{1,1}$-norm of $f$ on $(l_1,l_2)$, 
$\alpha$, and the $C^{1,\alpha}$-norm of $\psi^{Q}$ in $\overline{B_r(Q)\cap \Omega^+}$.
It follows from \eqref{IandII} and \eqref{Noteleq0}
that, for $\epsilon<\epsilon_2$,
\begin{align}\label{boundry2=}
 \widetilde{\psi}^Q \geq \widetilde{\Psi}^Q \qquad \text{on $\partial (B_{\varepsilon_*}(Q) \cap \Omega^+) \setminus B_R(Q_0)$}.
\end{align}

Let
\begin{align*}
  \epsilon_0 \defs \min\{ \epsilon_1^*, \epsilon_2\},
\end{align*}
where $\epsilon_1^*$ and $\epsilon_2$ are defined in \eqref{defsepsilon1} and \eqref{epsilon2=}, respectively.
Then it follows from \eqref{NPsiQ>0}, \eqref{boundary1=}, \eqref{boundry2=}, and the maximum principle that
\begin{align}\label{comparisonWcone}
 \widetilde{\psi}^Q - \widetilde{\Psi}^Q \geq 0 \qquad \text{in $B_{\varepsilon_*}(Q)\cap \Omega^+$}.
\end{align}
Since $(\widetilde{\psi}^Q-\widetilde{\Psi}^Q)(Q)=0$, \eqref{comparisonWcone} implies
\begin{align*}
 \frac{\dif \widetilde{\psi}^Q}{\dif {\mathbf{n}}}(Q)\geq  \frac{\dif \widetilde{\Psi}^Q}{\dif \mathbf{n}} (Q)= 2 l R \kappa e^{-lR^2} >0.
\end{align*}
This contradicts the fact that $\frac{\dif \widetilde{\psi}^Q}{\dif {\mathbf{n}}}(Q)=0$
by the definition of $\widetilde{\psi}^Q$ given by \eqref{psiQQIN} and \eqref{widetildepsiQ}.
Therefore, $(U,V)(Q)$ is a normal vector at the sonic point $Q$, \emph{i.e.}, the sonic point $Q$ is exceptional.
Finally, $\epsilon_0$ depends only on $a_2$, $r$, the $C^{1,1}$-norm of $f$ on $(l_1,l_2)$, 
$\alpha$, and the $C^{1,\alpha}$-norm of $\psi^{Q}$ in $\overline{B_r(Q)\cap \Omega^+}$,
because $\epsilon_1^*$ and $\epsilon_2$ depend only on these parameters.
\end{proof}

\section{Shock Reflection-Diffraction Problem with Non-Uniform State}
As an application of Theorems \ref{theorem1}--\ref{theoremcone},
we now study the sonic curves in the shock reflection-diffraction problem
with non-uniform incoming flow (see Fig. \ref{charac}).
In the self-similar coordinates $(\xi,\eta)$,
we consider the problem in the half-plane $\eta>0$ outside a wedge of angle $\theta_{\rm w}$, {\it i.e.}, in the domain:
\begin{align*}
  \Lambda\defs \{\xi\leq 0,\, \eta>0\}\cup \{\eta > \xi\tan\theta_{\rm w},\, \xi>0\}.
\end{align*}
The solutions satisfy the potential flow equation in the distributional sense in $\Lambda$,
the slip boundary condition on $\partial\Lambda$:
\begin{align}\label{4.2x}
  D\varphi \cdot \mbnu=0\qquad \text{on $\partial\Lambda$},
\end{align}
and the following asymptotic boundary condition at the infinity:
\begin{align*}
  \varphi \rightarrow \bar{\varphi} =
\begin{cases}
\varphi_0\quad \text{for $\xi>\xi_0, \eta<\xi\tan\theta_{\rm w}$},\\
\varphi_1\quad \text{for $\xi<\xi_0,\eta>0$},
\end{cases}
\qquad\mbox{as $\xi^2 + \eta^2 \rightarrow \infty$}.
\end{align*}

\begin{figure}[!h]
	\centering
	\includegraphics[width=0.35\textwidth]{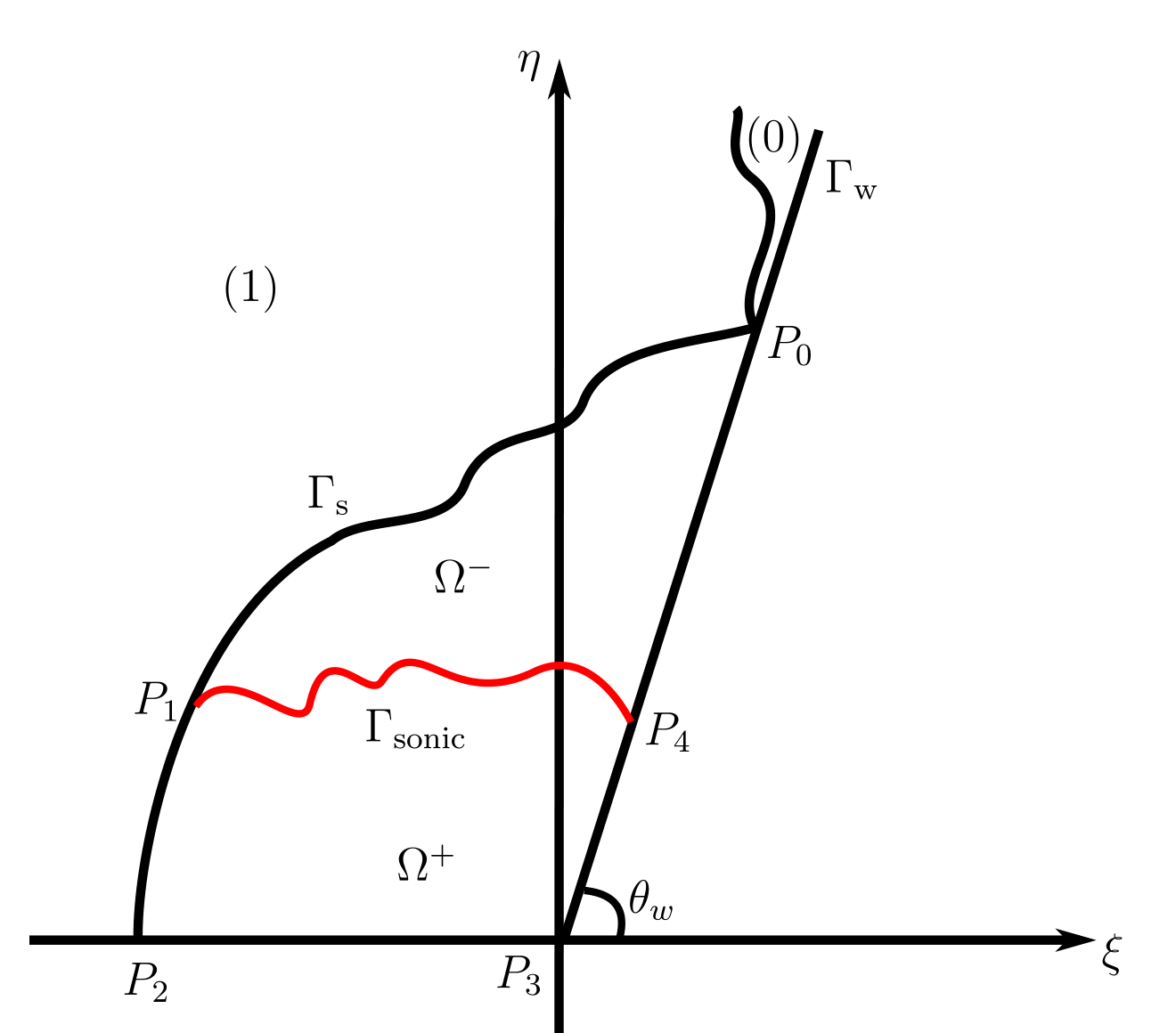}
	\caption{Shock reflection-diffraction problem with non-uniform states. \label{charac}}
\end{figure}

For the regular shock reflection with uniform upstream flow, which we refer to as uniform state (1),
the states along $\Gamma_{\rm sonic}$ are constant, as established
in \cite{Bae-Chen-Feldman, Chen-Feldman2,Chen-Feldman3,Chen-Feldman4,Chen-Feldman5,Chen-Feldman-Xiang},
and $D\varphi^+(Q)$ is the normal to the sonic arc at $Q\in\Gamma_{\rm sonic}$.
Let $\Omega^-$ and $\Omega^+$ denote, respectively, the pseudo-supersonic
and pseudo-subsonic regions of the solution $\varphi$ between the reflected
shock $P_0P_2$ and the wedge boundary $P_0P_3$ (see Fig. \ref{charac}).
The existence of a solution $\varphi$ of this problem was established in \cite{Chen-Feldman4}.
Such a solution consists of a uniform state $\varphi^-$ in $\Omega^-$ and a state $\varphi^+$ in $\Omega^+$,
satisfying
\begin{enumerate}
\item In $ \Omega^-$,
\begin{equation}\label{wsupassc22}
\mbox{$\qquad\,$ $\varphi^-$ is supersonic and satisfies Eqs. \eqref{self-similar}--\eqref{rhoeqself} with $\varphi^-\in C^2(\overline{\Omega^-}\backslash\{P_1\})$.}
\end{equation}
Denote $(\partial_{\xi} \varphi^-,\partial_{\eta} \varphi^-)=(U_-, V_-)$.

\smallskip
\item In $ \Omega^+ $,
\begin{equation}\label{subassc22b}
\begin{split}
&\mbox{ 
$\varphi^+$ is subsonic and satisfies Eqs. \eqref{self-similar}--\eqref{rhoeqself}} \\
&\qquad\mbox{ with $ \varphi^+\in C^2(\overline{\Omega^+}\backslash\{P_1,P_3\})
\cap C^{1,\alpha}(\overline{\Omega^+})$},
\end{split}
\end{equation}
for some $\alpha\in (0,1)$.
Denote $(\partial_{\xi} \varphi^+,\partial_{\eta} \varphi^+)=(U_+, V_+)$.

\smallskip
\item 
The sonic curve is connected.
On the sonic curve $\Gamma_{\rm sonic}$, the continuity conditions hold:
\begin{align}\label{conti2r}
  \varphi^+ = \varphi^-, \quad D \varphi^+  =D \varphi^- \qquad\,\, \text{on $\Gamma_{\rm sonic}$}.
\end{align}
\item On the shock front $\Gamma_{\rm shock}$, the following entropy conditions hold:
\begin{align}\label{entropyC}
 D \varphi_1 \cdot \mathbf{n}_{\rm s}> D\varphi\cdot \mathbf{n}_{\rm s}>0, \qquad \rho> \rho_1>0,
\end{align}
and the following Rankine-Hugoniot (R-H) conditions and Bernoulli law hold:
\begin{align}
  &\varphi = \varphi_1,\label{iRH1}\\
  &\rho D \varphi\cdot \mathbf{n}_{\rm s}  = \rho_1 D \varphi_1 \cdot \mathbf{n}_{\rm s},\label{iRH2}\\
  &\frac{|D\varphi|^2}{2} + \rho^{\gamma-1}  = \frac{|D\varphi_1|^2}{2} + \rho_1^{\gamma-1},\label{iRH3}
\end{align}
where $\varphi=\varphi^-$ in $\overline{\Omega^-}$ and $\varphi=\varphi^+$ on $\overline{\Omega}^+$.
\item $\Gamma_{\rm sonic}$ and $\Gamma_{\rm shock}$ meet at a unique point $P_1$ such that
\begin{equation}\label{4.11x}
\mbox{$P_1$ is not on $\Gamma_{\rm w}$ or the $\xi$--axis.}
\end{equation}
\item
The slip boundary condition \eqref{4.2x} holds on $\partial\Lambda=\Gamma_{\rm w}\cup\Gamma_{\rm sym}$,
where $\Gamma_{\rm w}$ is the wedge-boundary and $\Gamma_{\rm sym}$ is the symmetry line $\{\eta=0,\,\xi\leq0\}$.
\end{enumerate}

In this section, we analyze the sonic curve in this problem with a non-uniform state,
when the solutions are a small perturbation of the solution given in \cite{Chen-Feldman4}.
More precisely, 
let solution $\varphi$ of this problem satisfy
\eqref{4.2x} and \eqref{wsupassc22}--\eqref{4.11x}.
Then we establish two theorems as presented below. The first concerns the structure of the sonic curve 
based on the properties of the solution from the subsonic side.

\begin{thm}\label{thmappsub}
\begin{enumerate}[\rm (i)]

\item\label{thmappsub-i1}
Assume that conditions \eqref{wsupassc22}--\eqref{conti2r} hold, and fix
$\alpha\in(0,1)$. Let $Q\in \Gamma_{\rm sonic}^0$ be such that

\begin{itemize}
\item[(a)] There exist $r>0$ and $f\in C^{1,1}(\mathbb R)$ for which
\eqref{soniccurve2} is satisfied in $B_r(Q)$. By \eqref{subassc22b}, reducing $r$ if necessary, 
we may assume that $\varphi\in C^2(\overline{\Omega^+\cap B_r(Q)});$

\item[(b)] The sonic speed satisfies $c(Q)>0;$

\item[(c)] In the rotated coordinate system
$(\tilde{\xi},\tilde{\eta})$, whose $\tilde{\xi}$-axis is aligned with
$D\varphi^+(Q)$, there exist constants $a_2>0$ and $\epsilon\ge0$
such that
\begin{align}\label{assumeAthm2}
\varphi^+_{\tilde{\xi}\tilde{\xi}}(Q)+1\ge 4a_2>0,
\qquad
|\varphi^+_{\tilde{\xi}\tilde{\eta}}(Q)|
+
|\varphi^+_{\tilde{\eta}\tilde{\eta}}(Q)+1|
\le \frac{\epsilon}{4}.
\end{align}
\end{itemize}

Let $\theta_1$ denote the angle between $D\varphi^+(Q)$ and the inner
normal to $\Gamma_{\rm sonic}$ at $Q$. Let $\epsilon_0$ and
$\varepsilon_1$ be the constants provided by {\rm Theorem
\ref{theoremcone}}, determined by $a_2$, $r$, $\alpha$, the
$C^{1,1}$-norm of $f$, and
$\|\varphi\|_{C^{1,\alpha}(\overline{\Omega^+\cap B_r(Q)})}$.

If
\[
\epsilon<\min\{\epsilon_0,\,3a_2\}
\qquad\text{and}\qquad
|\theta_1|<\varepsilon_1,
\]
then $Q$ is an exceptional point.

\vspace{3pt}
\item\label{thmappsub-i2}
Assume that the hypotheses of part~\eqref{thmappsub-i1} are satisfied at every
point $Q\in\Gamma_{\rm sonic}^0$, where the constants $r$, $a_2$, and
$\varepsilon$ are allowed to depend on $Q$. Then the sonic curve
$\Gamma_{\rm sonic}$ is a circular arc.
\end{enumerate}
\end{thm}

\begin{rem}\label{rem-8-1}
For the regular shock reflection with uniform state $(1)$,
the solution in $\Omega^-$ is a uniform state, denoted by state $(2)$, with potential
$\varphi_2$.
Hence, $(U, V)=(u_2-\xi, v_2-\eta)$ in $\Omega^-$, where $u_2$ and $v_2$ are positive constants and $(u_2, v_2)\in P_3P_4$ by \eqref{4.2x}.
The sonic curve
$\Gamma_{\rm sonic}$ is therefore an arc of the circle centered at $(u_2, v_2)$ with radius $c_2$,
where $c_2$ is the constant speed of sound of state $(2)$.

Furthermore, by \eqref{conti2r}, $\varphi_2^Q=\varphi_2$ for every
$Q\in\Gamma_{\rm sonic}$, where $\varphi_2^Q$ is defined by
\eqref{phi-2-Q}. Defining
\[
\psi:=\varphi-\varphi_2,
\]
it follows that $\psi^Q=\psi$ for every $Q\in\Gamma_{\rm sonic}$, where
$\psi^Q$ is defined by \eqref{psiQQIN}. In addition, \eqref{conti2r}
implies that $D\psi=0$ on $\Gamma_{\rm sonic}$.

Denote by $(r, \theta)$ the polar coordinates with respect to
the center  $(u_2, v_2)$ of the arc $\Gamma_{\rm sonic}$. It follows from 
 {\rm \cite{Bae-Chen-Feldman}} that
\begin{equation}\label{secondDerivOnSonic-Unpert}
\psi_{rr}=\frac{1}{\gamma+1},\quad \psi_{r\theta}=\psi_{\theta\theta}=0
\qquad\mbox{on }\; \Gamma_{\rm sonic}.
\end{equation}
Since $D\varphi=D\varphi_2$ on $\Gamma_{\rm sonic}$, the gradient
$D\varphi$ is directed along the radial direction there. Therefore,
for any $Q\in\Gamma_{\rm sonic}^0$, working in the coordinates
$(\tilde\xi,\tilde\eta)$ introduced in
{\rm Theorem~\ref{thmappsub}\eqref{thmappsub-i1}}, we obtain from
\eqref{secondDerivOnSonic-Unpert} and the identity $D\psi=0$ on
$\Gamma_{\rm sonic}$ that
\begin{equation*}
\varphi_{\tilde\xi\tilde\xi}
=
\frac{1}{\gamma+1}-1,
\quad\,\,
\varphi_{\tilde\xi\tilde\eta}=0,
\quad\,\,
\varphi_{\tilde\eta\tilde\eta}=-1\qquad\,\,\, \mbox{at $Q$}.
\end{equation*}

Consequently, the hypotheses of
{\rm Theorem~\ref{thmappsub}\eqref{thmappsub-i1}} are satisfied at every
point $Q\in\Gamma_{\rm sonic}^0$ with
\[
a_2=\frac{1}{4(\gamma+1)},
\qquad
\varepsilon=0.
\]
Therefore, {\rm Theorem~\ref{thmappsub}\eqref{thmappsub-i2}} implies that
$\Gamma_{\rm sonic}$ is a circular arc, in agreement with both the
results of the aforementioned references and experimental observations.
\end{rem}

\begin{rem}
As shown in {\rm Remark \ref{rem-8-1}},
the standard regular reflection solution with a uniform incident state
constructed in
{\rm \cite{Bae-Chen-Feldman, Chen-Feldman2, Chen-Feldman3,
Chen-Feldman4, Chen-Feldman5, Chen-Feldman-Xiang}}
satisfies the assumptions of {\rm Theorem~\ref{thmappsub}}. Therefore,
{\rm Theorem~\ref{thmappsub}} implies that any sufficiently small perturbation
of this solution for which the assumptions of
{\rm Theorem~\ref{thmappsub}\eqref{thmappsub-i1}} remain valid at every $Q\in \Gamma_{\rm sonic}^0$ 
 should also have a sonic curve that is a
circular arc.
\end{rem}

\begin{rem}\label{rem4.3x}
As indicated in {\rm Remark \ref{rem2.1x}}, based on {\rm Theorem \ref{theorem1}},
we may apply {\rm Theorem 3.1} in {\rm \cite[Page 517]{Bae-Chen-Feldman}}.
Then, if a solution exists with a non-uniform upstream flow that is a small perturbation of the standard regular reflection solution,
then the sonic curve must be an arc, the velocity and the density on the sonic arc coincide with those of the uniform state determined by this arc
{\rm (}which we call it state $(2)${\rm )}, and the solution is $C^{2,\alpha}$ in $\Omega^+$ up to the sonic arc, except at the point $P_1$.

In fact, by {\rm Theorem \ref{thmappsub}}, all sonic points are exceptional. 
Consequently,
$\Gamma_{\rm sonic}$ is an arc, and the vector field $(U,V)$ is directed along the radius of the circle containing $\Gamma_{\rm sonic}$.
Let $(u_2, v_2)$ denote the center of this circle.
Then $(U, V)$ can be written as
\begin{align}\label{UV=}
  (U,V) = \lambda(u_2-\xi, v_2-\eta)
\end{align}
for some scalar function $\lambda$.
By {\rm Theorem \ref{theorem1}},
at every point $Q\in\Gamma_{\rm sonic}$,
either $c(Q)=0$ or $\partial_{\boldsymbol{\tau}}c(Q)=0$.
Since $c>0$ on $\Gamma_{\rm sonic}$, it follows that $c=c_2$ on $\Gamma_{\rm sonic}$
for some constant $c_2>0$.
Substituting \eqref{UV=} into $U^2+V^2=c_2^2$ on $\Gamma_{\rm sonic}$ yields
 \begin{align}\label{circle8.12x}
 c^2_2=(u_2-\xi)^2 + (v_2-\eta)^2 = \frac{c_2^2}{\lambda^2}\qquad\mbox{on }\Gamma_{\rm sonic}.
 \end{align}
By {\rm Theorem \ref{theorem1}}, $c_2$ is precisely the radius of the sonic arc $\Gamma_{\rm sonic}$.
Thus,
 $\lambda=1$ and $  (U,V) = (u_2-\xi, v_2-\eta)$.
It follows that $\partial_{\tau}\varphi^+=(U,V)\cdot\tau=0$ on $\Gamma_{\rm sonic}$, which implies that $\varphi^+$ is a constant
on $\Gamma_{\rm sonic}$.
Define
\begin{align}\label{defpsix}
\psi\defs \varphi^+-\big(-\frac{1}{2} (\xi^2 +\eta^2) + u_2\xi + v_2\eta+K_0\big) \qquad\mbox{in $\Omega^{+}$},
\end{align}
where constant $K_0$ is chosen so that $\psi=0$ on $\Gamma_{\rm sonic}$.
It follows directly that $\psi=0$ and $D\psi=0$ on $\Gamma_{\rm sonic}$.
Let $(\xi,\eta) = (r\cos\theta, r\sin\theta)$ and $(x,y)=(c_2 - r,\theta)$.
Then $\Gamma_{\rm sonic}\subset\{(x,y)\,:\, x=0,\,y\in(l_3,l_4)\}$ for some constants $l_3$ and $l_4$.
Suppose that $\psi$ satisfies assumptions {\rm (3.3)}, {\rm (3.6)--(3.7)}, and {\rm (3.9)} of {\rm Theorem 3.1}
in {\rm \cite[Page 517]{Bae-Chen-Feldman}}. Then
 $\psi\in C^{2,\alpha}((\Omega^+\cup\Gamma_{\rm sonic} )\backslash \{P_1\}))$
and the one-sided derivatives from the subsonic side are $\psi_{xx}(0,y)=\frac{1}{\gamma+1}$
and $\psi_{xy}(0,y)=\psi_{yy}(0,y)=0$ for $y\in(l_3,l_4)$.
In fact, assumption {\rm (3.3)} in {\rm \cite{Bae-Chen-Feldman}} is verified in {\rm Lemma \ref{lem4.1xw}} below.
\end{rem}

The second theorem is related to the hyperbolic region.
Before stating the theorem, we define $\bar{\partial}^{\pm}$ as
\begin{align}\label{barpartial+-}
  \bar{\partial}^+ \defs  \cos\alpha\,\partial_{\xi }+\sin \alpha\,\partial_{\eta},
  \qquad \bar{\partial}^- \defs \cos\beta\,\partial_{\xi }+\sin \beta\,\partial_{\eta},
\end{align}
where
\begin{align}\label{cangle}
  \tan\alpha \defs \frac{UV + c\sqrt{U^2 +V^2 -c^2}}{U^2 -c^2},\qquad \tan \beta \defs  \frac{UV - c\sqrt{U^2 +V^2 -c^2}}{U^2 -c^2}
\end{align}
are the eigenvalues of equations \eqref{1}--\eqref{1.1}, \emph{i.e.},
 \begin{align}\label{reeq}
   (u_{\xi},v_{\xi})^{\top} + A(U,V,c)(u_{\eta},v_{\eta})^{\top} = (0,0)^{\top},
 \end{align}
where
\[
A(U,V,c) = \begin{pmatrix}
  -\frac{2UV}{c^2 -U^2} & \frac{c^2 -V^2}{c^2-U^2}  \\
  -1& 0
\end{pmatrix}.
\]

\begin{thm}\label{thmapp}
For any solution satisfying \eqref{4.2x} and conditions \eqref{wsupassc22}--\eqref{4.11x},
and for any sonic point $Q= (\xi_*,f(\xi_*))$,
let ${\psi}^Q$ be defined in \eqref{psiQQIN}.
Assume that either $\bar{\partial}^{\pm} {c}\geq 0$ or $\bar{\partial}^{\pm} {c}\leq 0$
on $\Gamma_{\rm shock}\cap\overline{\Omega^-}$.
Assume further that there exist a constant $0<a_1 < \frac{1}{\gamma+1}$ and a sufficiently small constant $\varepsilon>0$, both independent of $Q$,
such that
\begin{align}\label{Qnpsi<}
\psi^Q(Q + t\mathbf{n}_Q) < a_1 \big(\frac{\mathbf{n}_Q\cdot(U,V)(Q)}{c}\big)^2t^2,
\end{align}
for all $t \in (0, \varepsilon)$,
where $\mathbf{n}_Q$ denotes the unit normal to $\Gamma_{\rm sonic}$ at $Q$ pointing into $\Omega^-$.
Then the sonic curve $P_1P_4$ shown in {\rm Fig. \ref{charac}} must lie on a circle.
\end{thm}

\begin{rem}
For the regular shock reflection with uniform state $(1)$,
\emph{i.e.}, when $(u_-, v_-, \rho_-)$ are constant so that $\psi^{Q}\equiv0$,
the corresponding solution has been analyzed
in {\rm \cite{Bae-Chen-Feldman, Chen-Feldman2,Chen-Feldman3,Chen-Feldman4,Chen-Feldman5,Chen-Feldman-Xiang}}.
In this case, $\partial^\pm c=0$ on $\Gamma_{\rm shock}\cap\overline{\Omega^-}$,
since the derivatives of the sonic speed along
the characteristic directions vanish throughout $\Omega^-$.
Therefore, all the hypotheses of {\rm Theorem \ref{thmapp}} are satisfied.
It follows that $\Gamma_{\rm sonic}$ must be a circular arc, which is consistent with both the results established
in the aforementioned references and the corresponding experimental observations.
\end{rem}

We now prove Theorems \ref{thmappsub}--\ref{thmapp} in order in this section.
First, Theorem  \ref{thmappsub} is proved based on Lemma \ref{arc} and Theorem \ref{theoremcone}.

\begin{proof}[Proof of {\rm Theorem \ref{thmappsub}}]
We first prove part \eqref{thmappsub-i1} of the Theorem.

We show that the sonic point $Q$ is exceptional by applying Theorem \ref{theoremcone}.
In order to apply Theorem \ref{theoremcone}, it suffices to show \eqref{assumeAthmcone},
which takes form \eqref{assumeAthm2x} in the coordinates
introduced in \eqref{defcoordinate}. Thus, it suffices to show \eqref{assumeAthm2x}.

It is direct to see that the  $(\tilde{\xi},\tilde{\eta})$--coordinates given in Theorem \ref{thmappsub}
are the same as the ones introduced in \eqref{defcoordinate}.
Thus, for $\widetilde{\psi}^Q$ defined by \eqref{widetildepsiQ}, it follows from \eqref{assumeAthm2} that the derivatives
from the $\Omega^+$--side satisfy
\begin{equation*}
 \widetilde{\psi}^Q_{\tilde{\xi}\tilde{\xi}}(Q)\geq 4a_2 >0,\qquad\,
 |\widetilde{\psi}^Q_{\tilde{\xi}\tilde{\eta}}(Q)|+| \widetilde{\psi}^Q_{\tilde{\eta}\tilde{\eta}}(Q)|\leq \frac{\epsilon}{4}.
\end{equation*}
Since $\varphi\in C^2(\overline{\Omega^+\cap B_r(Q)})$, as noted in the statement of the theorem, it follows that
$\widetilde{\psi}^Q\in C^2(\overline{\Omega^+\cap B_r(Q)})$.
Therefore, $\widetilde{\psi}^Q$ admits a $C^2$-extension across the sonic curve to a function in
$C^2(\overline{B_r(Q)})$.
In the remainder of the proof, we work with a fixed such extension.
Then there exists a constant $\tilde{\varepsilon}<\varepsilon$ small such that,
for each $\widehat{Q}\in \overline{(B_{\tilde{\varepsilon}}(Q))}$,
\begin{equation*}
 \widetilde{\psi}^Q_{\tilde{\xi}\tilde{\xi}}(\widehat{Q})\geq 3a_2 >0,\qquad
 | \widetilde{\psi}^Q_{\tilde{\xi}\tilde{\eta}}(\widehat{Q})|+| \widetilde{\psi}^Q_{\tilde{\eta}\tilde{\eta}}(\widehat{Q})|\leq \frac{\epsilon}{3}.
\end{equation*}

By \eqref{psiQQIN}, we see that $\widetilde{\psi}^Q(Q)=0$ and $D\widetilde{\psi}^Q(Q)=0$.
Thus, for each $(\tilde\xi, \tilde\eta) \in B_{\tilde \varepsilon}(Q)\cap \Omega^+$ there exist $\widehat{Q}\in {B_{\tilde{\varepsilon}}(Q)}$ depending on $(\tilde{\xi},\tilde{\eta})$ such that
\begin{align*}
 \widetilde{\psi}^Q(\tilde{\xi},\tilde{\eta})
 &= \frac12 \widetilde{\psi}_{\tilde{\xi}\tilde{\xi}}^Q(\widehat{Q}) (\tilde{\xi} - \tilde{\xi}_*)^2
 + \frac12  \widetilde{\psi}_{\tilde{\eta}\tilde{\eta}}^Q (\widehat{Q})(\tilde{\eta} - \tilde{\eta}_*)^2
 +  \widetilde{\psi}_{\tilde{\xi}\tilde{\eta}}^Q (\widehat{Q})(\tilde{\xi} - \tilde{\xi}_*)(\tilde{\eta}- \tilde{\eta}_*)\\
&\geq \frac32 a_2 (\tilde{\xi} - \tilde{\xi}_*)^2 - \frac16\epsilon  (\tilde{\eta} - \tilde{\eta}_*)^2
-\frac{1}{3}\epsilon |(\tilde{\xi} - \tilde{\xi}_*)(\tilde{\eta}- \tilde{\eta}_*)|\\
&\geq (\frac32 a_2 -\frac16\epsilon) (\tilde{\xi} - \tilde{\xi}_*)^2 - \frac13\epsilon  (\tilde{\eta} - \tilde{\eta}_*)^2 \\
&> a_2 (\tilde{\xi} - \tilde{\xi}_*)^2 - \epsilon  (\tilde{\eta} - \tilde{\eta}_*)^2,
\end{align*}
if $\frac12 a_2 -\frac16\epsilon>0$. Therefore, \eqref{assumeAthm2x} holds if $a_2>\frac13\epsilon$.

This completes the proof of part \eqref{thmappsub-i1} of the Theorem.

\medskip
We now prove part \eqref{thmappsub-i2}. Since the hypotheses of
part \eqref{thmappsub-i1} hold at every point
$Q\in\Gamma_{\rm sonic}$, it follows from
part \eqref{thmappsub-i1} that every point of
$\Gamma_{\rm sonic}$ is exceptional. 
Therefore, the conclusion follows from the final assertion of Lemma \ref{arc}.
\end{proof}

\begin{figure}[!h]
	\centering
	\includegraphics[width=0.45\textwidth]{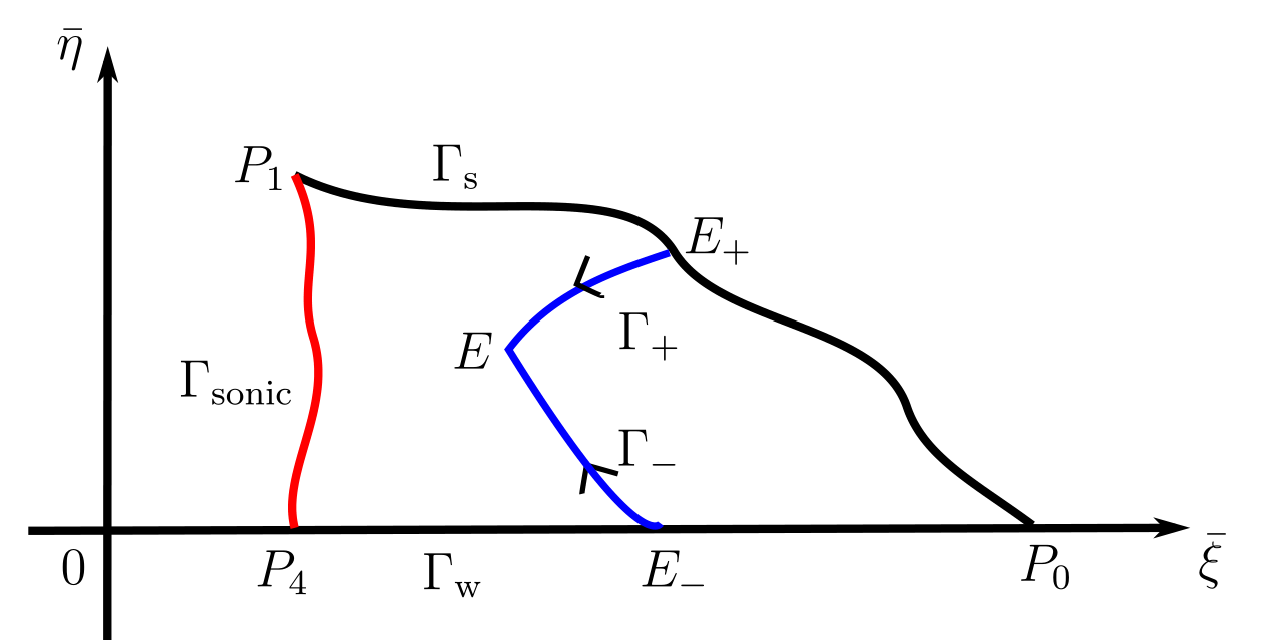}
	\caption{The characteristics in the supersonic region.\label{C+-}}
\end{figure}

Next, for Theorem \ref{thmapp}, we consider the solutions in the supersonic domain $\Omega^-$.
Since equations \eqref{self-similar}--\eqref{rhoeqself} are invariant under the rotation of the coordinates,
we choose the coordinates $(\bar{\xi}, \bar{\eta})$ as shown in Fig. \ref{C+-}.
By assumptions in Theorem \ref{thmapp}, we know  that $\bar{\partial}^{\pm} c\geq 0$ or $\bar{\partial}^{\pm} c\leq 0$ on $\Gamma_{\rm s}$.
Then we have
the following lemma:

\begin{lem}\label{+-barc>=0}
 Under the assumptions given in {\rm Theorem \ref{thmapp}},
  \begin{align}\label{pmdelta-+}
  \mbox{either}\,\,\,\,\, \bar{\partial}^{\pm} c\geq 0\,\, \text{in $\Omega^-$} \,\,\,\,\, \mbox{or} \,\,\,\,\,\, \bar{\partial}^{\pm} c\leq 0\,\, \text{in $\Omega^-$}.
  \end{align}
\end{lem}

\begin{proof}
For any point $E\in \Omega^-$, let $\Gamma_+$ be the positive characteristic curve passing through $E$ and
intersecting with $\Gamma_{\rm s}$ at point $E_+$,
and let $\Gamma_-$ be the negative characteristic curve passing through $E$ and intersecting with $\Gamma_{\rm w}$ at point $E_-$.
Choose the direction of the positive characteristic curve $\Gamma_+$ from $E_+$ to $E$,
and the direction of the negative characteristic curve $\Gamma_-$ from $E_-$ to $E$.
A direct calculation yields
\begin{align}
  c \bar{\partial}^+ \bar{\partial}^- c &= \bar{\partial}^- c
  \big( \sin 2\delta + \frac{\gamma +1}{2(\gamma-1)\cos^2\delta} \bar{\partial}^- c + \frac{\gamma+1 - 2\sin^2 2\delta}{2(\gamma-1)\cos^2\delta}\bar{\partial}^+ c\big),\label{c+-}\\
   c \bar{\partial}^- \bar{\partial}^+ c &= \bar{\partial}^+ c
  \big( \sin 2\delta + \frac{\gamma +1}{2(\gamma-1)\cos^2\delta} \bar{\partial}^+ c + \frac{\gamma+1 - 2\sin^2 2\delta}{2(\gamma-1)\cos^2\delta}\bar{\partial}^- c\big),\label{c-+}
\end{align}
where $\sin\delta\defs \frac{c}{q}$ and $\delta \in(0, \frac{\pi}{2})$ in the supersonic region ({\it cf.} (2.27) in \cite{LaiSheng}).
Let
$$
\mathcal{A}\defs  \frac{1}{c}\big( \sin 2\delta + \frac{\gamma +1}{2(\gamma-1)\cos^2\delta} \bar{\partial}^- c
+ \frac{\gamma+1 - 2\sin^2 2\delta}{2(\gamma-1)\cos^2\delta}\bar{\partial}^+ c\big).
$$

Without loss of generality, we choose the sign of $\partial^\pm c$ on $\Gamma_s$ as $\bar{\partial}^{\pm} {c}\geq 0$ on $\Gamma_{\rm s}\cap\overline{\Omega^-}$.
Then it follows from \eqref{c+-} that
\begin{align}\label{Einter}
  \bar{\partial}^- c(E) =  e^{\int_{E_+}^E \mathcal{A}\,\dif S } \bar{\partial}^- c(E_+)\geq 0,
\end{align}
where $\int_{E_+}^E (\,\cdot\,)\,\dif S$ represents the integration along the positive characteristic
curve $\widetilde{E_+ E}$ from $E_+$ to $E$.
By the arbitrariness of point $E\in\Omega^-$, we see that
 $\bar{\partial}^- c\geq 0$ in $\Omega^-$ if $\bar{\partial}^- c\geq 0$ on $\Gamma_{\rm s}$.

Next, we consider the sign of $\bar{\partial}^+c$ in $\Omega^-$.
By the slip boundary condition \eqref{4.2x} on $\Gamma_{\rm w}$, \emph{i.e.}, $\frac{V}{U} = \tan\theta_{\rm w}$
with constant $\theta_{\rm w}$, we have
\begin{align}\label{Gammaw=}
  (U \sigma_{\xi} + V\sigma_{\eta})|_{\Gamma_{\rm w}} =0,
\end{align}
where  $\sigma$ is defined in \eqref{4.7x}.
Then  \eqref{perpc==} and \eqref{Gammaw=} yield
\begin{align}\label{perpCGw=0}
  {\partial}^{\perp} c|_{\Gamma_{\rm w}}  =0.
\end{align}

By \eqref{barpartial+-}, we have
\begin{align}\label{+-ccperpcx}
\bar{\partial}^+ c  - \bar{\partial}^- c
&= (\cos\alpha - \cos\beta)c_{\xi}  + (\sin \alpha - \sin \beta)c_{\eta}\notag\\
&= 2 \sin(\frac{\alpha-\beta}{2})\big( \cos(\frac{\alpha+\beta}{2}) c_{\eta} -  \sin(\frac{\alpha +\beta}{2})c_{\xi}\big).
\end{align}
By \eqref{cangle}, we obtain
\begin{align}
  \tan(\alpha + \beta) &= \frac{\tan\alpha +\tan\beta}{1-\tan\alpha \tan\beta} 
=\frac{2UV}{U^2 -V^2},\label{tanalpha+beta=}\\
   \tan(\alpha - \beta) &= \frac{\tan\alpha - \tan\beta}{1+ \tan\alpha \tan\beta} 
= \frac{2c\sqrt{U^2 +V^2 -c^2}}{U^2 + V^2 - 2c^2}.\label{tanalpha-beta}
\end{align}
Then, by straightforward calculation, we have
\begin{align}
\cos(\frac{\alpha +\beta}{2}) =& \sqrt{\frac{1+\cos(\alpha +\beta)}{2}} 
= \frac{U}{q},\label{cosalpha+beta2}\\
\sin(\frac{\alpha +\beta}{2}) =&\sqrt{\frac{1-\cos(\alpha +\beta)}{2}} 
= \frac{V}{q},\\
\sin(\frac{\alpha -\beta}{2}) = &\sqrt{\frac{1-\cos(\alpha -\beta)}{2}} 
= \frac{c}{q}. \label{sinalpha-beta2}
\end{align}

We note that only the positive sign + is considered in \eqref{cosalpha+beta2}--\eqref{sinalpha-beta2}. This is justified as follows:
Without loss of generality, as shown in Fig. \ref{self}, $\alpha$ and $\beta$ are defined as the angles between the $\Gamma_{\pm}$
characteristic directions and the positive $\xi$-axis,
where $\alpha$ and $\beta$ satisfy $0 < \alpha - \beta < \pi$.
That is, the direction of the $\Gamma_{\pm}$ characteristics is defined so that the Mach angle $\delta\defs \frac{\alpha-\beta}{2}$ is an acute angle.
By direct computation, it is direct to see that $\sigma \defs \frac{\alpha+\beta}{2}$ is the angle between the direction
of velocity $(U,V)$ and the positive $\xi$-axis.
\begin{figure}[!h]
	\centering
	\includegraphics[width=0.45\textwidth]{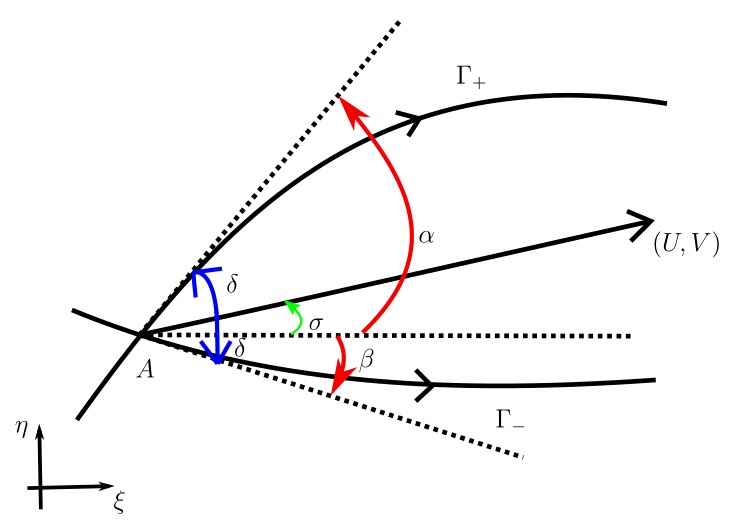}
	\caption{Characteristic directions and characteristic angles. \label{self}}
\end{figure}

Therefore, \eqref{+-ccperpcx} and \eqref{cosalpha+beta2}--\eqref{sinalpha-beta2} imply that
\begin{align}\label{+-ccperpc}
\bar{\partial}^+ c  - \bar{\partial}^- c
=&2 \frac{c}{q}\big( \frac{U}{q} c_{\eta} -  \frac{V}{q} c_{\xi}\big)
= -\frac{2c}{q^2} {\partial}^{\perp}c.
\end{align}
Thus, \eqref{perpCGw=0}, \eqref{+-ccperpc}, and $\bar{\partial}^- c\geq 0$ in $\Omega^-$ imply
\begin{align}\label{+cGW0}
  \bar{\partial}^+c|_{\Gamma_{\rm w}}\geq 0.
\end{align}
Then, applying \eqref{c-+} and \eqref{+cGW0}, we can deduce that $\bar{\partial}^+c\geq 0$ in $\Omega^-$.

Similar arguments yield that $\bar{\partial}^{\pm} c\leq 0$ in $\Omega^-$ if $\bar{\partial}^{\pm} c\leq 0$ on $\Gamma_{\rm s}$.
\end{proof}

We are now ready to prove Theorem \ref{thmapp}, based on Theorems \ref{theorem1} and \ref{theoremnew}.

\begin{proof}[Proof of {\rm Theorem \ref{thmapp}}]
Let $\delta=\frac{\alpha-\beta}{2}$ and $\sigma=\frac{\alpha+\beta}{2}$. 
Because $q=c$ on $\Gamma_{\rm sonic}$, \eqref{sinalpha-beta2} yields $\delta =\frac{\pi}{2}$  on $\Gamma_{\rm sonic}$.
Then $\alpha=\sigma + \frac{\pi}{2}$ and $\beta=\sigma -\frac{\pi}{2}$ on $\Gamma_{\rm sonic}$, so that
\begin{align}
\bar{\partial}^+|_{\Gamma_{\rm sonic}}&=
 (\cos(\sigma + \frac{\pi}{2})\partial_{\xi} + \sin(\sigma + \frac{\pi}{2})\partial_{\eta})\mid_{\Gamma_{\rm sonic}}
= -\frac{1}{c} \partial^{\perp}|_{\Gamma_{\rm sonic}},\label{+perpsc}\\
\bar{\partial}^-|_{\Gamma_{\rm sonic}}
&= (\cos(\sigma - \frac{\pi}{2})\partial_{\xi} + \sin(\sigma - \frac{\pi}{2})\partial_{\eta})\mid_{\Gamma_{\rm sonic}}
=  \frac{1}{c}\partial^{\perp}|_{\Gamma_{\rm sonic}}.\label{-perpsc}
\end{align}
Then we have
\begin{align}
-\bar{\partial}^+|_{\Gamma_{\rm sonic}}=\bar{\partial}^-|_{\Gamma_{\rm sonic}} = \frac{1}{c}\partial^{\perp}|_{\Gamma_{\rm sonic}}.
\end{align}

Thus, by Lemma \ref{+-barc>=0} and the fact that
$c\in C^1(\overline{\Omega^-})$,
we have
  \begin{align*}
  {\partial}^{\perp}c =0\qquad \text{on $\Gamma_{\rm sonic}$}.
  \end{align*}
Applying Theorem \ref{theoremnew} and \eqref{Qnpsi<} for any point $Q\in\Gamma_{\rm sonic}$,
we know that $(U,V)(Q)$ is a normal to $\Gamma_{\rm sonic}$ at $Q\in \Gamma_{\rm sonic}$.
By Theorem \ref{theorem1} and the arbitrariness of the sonic point $Q$,
we conclude that $\Gamma_{\rm sonic}$ lies on a circle.
\end{proof}

\appendix
\section{}
It follows from Theorem \ref{thmappsub} (or Theorem \ref{thmapp}) and Remark \ref{rem4.3x} that
the optimal regularity of solutions across $\Gamma_{\rm sonic}$ given in Theorem 3.1
and Theorem 4.1 in \cite{Bae-Chen-Feldman} holds in the present case.
In fact, we can show assumption $(3.3)$ in \cite[Page 516]{Bae-Chen-Feldman} by the maximum principle via the following lemma:

\begin{lem}\label{lem4.1xw}
Under the assumptions given in {\rm Theorem \ref{thmappsub}} or {\rm Theorem \ref{thmapp}},
for $\psi$ defined by \eqref{defpsix} in {\rm Remark \ref{rem4.3x}}, $\psi$ satisfies the equation of the form:
\begin{equation}\label{A1}
(2x-a\psi_x+O_1)\psi_{xx}+O_2\psi_{xy}+(b+O_3)\psi_{yy}-(1+O_4)\psi_x+O_5\psi_y=0.
\end{equation}
Moreover, we assume that, some constants $N>0$, $M>0$, and $\beta\in(0,1)$,
\begin{align}
&\frac{|O_1(x,y)|}{x^2},\,\frac{|O_k(x,y)|}{x}\leq N\qquad\mbox{for }k=2,\cdots,5,\label{A2a}\\
&\frac{|DO_1(x,y)|}{x},\,|DO_k(x,y)|\leq N\qquad\mbox{for }k=2,\cdots,5,\label{A3}\\
&-Mx\leq\psi_x\leq\frac{2-\beta}{a}x.\label{A4}
\end{align}
If $\psi\geq0$ on the transonic shock $P_1P_2$, then
  \begin{align}\label{psi>0}
    \psi>0\qquad \mbox{in $\Omega^+$}.
  \end{align}
\end{lem}

\begin{rem}
In fact, condition $\psi\geq0$ on the transonic shock $P_1P_2$ is property {\rm (2.6.3)}
in {\rm \cite[Page 23]{Chen-Feldman4}},
which is crucial in the definition of the admissible solution in {\rm Definition 8.1.1} in {\rm \cite{Chen-Feldman4}}.
\end{rem}

\begin{proof} If the assumptions in Theorem \ref{thmappsub}
or the assumptions in Theorem \ref{thmapp} hold, $\Gamma_{\rm sonic}$ is an arc and all the sonic points are exceptional.
By Remark \ref{rem4.3x}, $u_2$ and $v_2$ are constants, and
\begin{align}\label{psi=0}
  \psi =0 \qquad \text{on $\Gamma_{\rm sonic}$}.
\end{align}
By \eqref{4.2x}, we have
\begin{align}\label{Dpsinw}
  D\psi \cdot \mathbf{n}_{\rm w} =0 \qquad\mbox{on $\Gamma_{\rm w}$}.
\end{align}
On the symmetry line $\{\eta=0, \xi<0\}$, it follows from \eqref{4.2x} that
\begin{align}\label{psieta<0}
  \psi_{\eta} =
-u_2 \tan\theta_{\rm w} <0\qquad\,  \text{on $\{\eta=0, \xi<0\}$}.
\end{align}

By assumptions \eqref{A2a}--\eqref{A4}, equation \eqref{A1} is elliptic
in $\overline{\Omega^+}\backslash\Gamma_{\rm sonic}$.
Therefore, by conditions \eqref{psi=0}--\eqref{psieta<0}, the assumption that $\psi\geq0$ on the transonic shock,
the strong maximum principle, and the Hopf lemma,
\eqref{psi>0} holds.
\end{proof}

\bigskip
\noindent
{\bf Acknowledgments.}
The research of
Gui-Qiang G. Chen was supported in part by the UK Engineering and Physical Sciences Research Council Awards
EP/V008854/1 and EP/V051121/1.
The research of Mikhail Feldman was supported in part by the National Science Foundation Under DMS-2054689   and DMS-2219391, and Steenbock Professorship Award at UW-Madison.
The research of Wei Xiang was supported in part by the
Research Grants Council of the HKSAR, China (Project No. CityU 11300021, CityU 11311722, CityU 11305523, and CityU 11305625).

\end{document}